\documentclass[11pt,a4paper,leqno]{article}
\pdfoutput=1

\usepackage[latin1]{inputenc}
\usepackage[english]{babel}
\usepackage[left=2.75cm, right=2.75cm, top=2.75cm, bottom=2.75cm]{geometry}

\usepackage{authblk}

\usepackage[titletoc,title]{appendix}

\usepackage{amsmath}
\usepackage{amsfonts}
\usepackage{amssymb}
\usepackage{amsthm}
\usepackage{bbm}		
\usepackage{mathtools}	
\usepackage{xcolor}
\usepackage{tikz}

\usepackage{graphicx}
\usepackage[suffix=]{epstopdf}
\usepackage{booktabs}
\usepackage[font=small,labelfont=bf]{caption} 

\definecolor{thelinkcolor}{RGB}{0,0,150}
\usepackage{hyperref}
\usepackage[nameinlink]{cleveref}
\crefname{equation}{}{}
\crefname{theorem}{theorem}{theorems}
\crefname{example}{example}{examples}
\crefname{lemma}{lemma}{lemmas}
\crefname{proposition}{proposition}{propositions}
\crefname{figure}{figure}{figures}
\crefname{table}{table}{tables}
\Crefname{equation}{}{}
\Crefname{theorem}{Theorem}{Theorems}
\Crefname{example}{Example}{Examples}
\Crefname{lemma}{Lemma}{Lemma}
\Crefname{proposition}{Proposition}{Proposition}
\Crefname{figure}{Figure}{Figures}
\Crefname{table}{Table}{Tables}

\hypersetup{colorlinks=true,
	urlcolor=blue,
	linkcolor=blue,
	citecolor=blue,
	bookmarksdepth=paragraph,
	pdftitle={Bounding extreme events in nonlinear dynamics (arXiv version)},
	pdfauthor={G. Fantuzzi and D. Goluskin},}      

\usepackage[shortlabels]{enumitem}
\setlist[itemize]{itemindent=0ex,itemsep=-0.5ex,leftmargin=2ex,topsep=5pt}
\setlist[enumerate]{label={\arabic*)},leftmargin=*,nolistsep,noitemsep,topsep=1ex}

\newcommand{\given}{;\,}								
\newcommand{\dx}{\,{\rm d}x}								
\newcommand{\dxi}{\,{\rm d}\xi}								
\newcommand{\maxphi}{\Phi^*}	
\newcommand{\maxphiinf}{\maxphi_\infty}
\newcommand{\maxphiT}{\maxphi_T}
\newcommand{\abs}[1]{\left\vert #1 \right\vert}				
\newcommand{\ddt}{\frac{{\rm d}}{{\rm d}t}}					

\newcommand\xqed[1]{\leavevmode\unskip\penalty9999 \hbox{}\nobreak\hfill\quad\hbox{#1}}
\newcommand\markendexample{\xqed{$\triangleleft$}}

\newcommand{\mosek}{{MOSEK}}
\newcommand{\sdpagmp}{{SDPA-GMP}}
\newcommand{\yalmip}{{YALMIP}}
\newcommand{\spotless}{{SPOT\sc{less}}}

\definecolor{mygreen}{RGB}{77,175,74}
\definecolor{myred}{RGB}{228,26,28}
\definecolor{myblue}{RGB}{55,126,184}
\definecolor{matlabgray}{RGB}{127,127,127}
\definecolor{matlabblue}{RGB}{0,113,188}
\definecolor{matlabred}{RGB}{216,82,24}
\definecolor{matlaborange}{RGB}{255,153,102}
\definecolor{matlabpurple}{rgb}{0.4940,0.1840,0.5560}
\definecolor{matlabgreen}{rgb}{0.0000,0.4980,0.0000}
\definecolor{matlabsafered}{RGB}{215,25,28}
\definecolor{matlabsafegreen}{RGB}{171,221,164}
\definecolor{lightbrown}{RGB}{149,99,99}
\definecolor{darkbrown}{RGB}{82,48,48}
\newcommand\solidrule[1][10pt]{\rule[0.5ex]{#1}{1.5pt}}
\newcommand\dashedrule{\mbox{%
		\solidrule[2pt]\hspace{2pt}\solidrule[2pt]\hspace{2pt}\solidrule[2pt]}}

\newcommand\dottedrule{\mbox{%
			\solidrule[1pt]\hspace{1pt}\solidrule[1pt]\hspace{1pt}\solidrule[1pt]\hspace{1pt}\solidrule[1pt]\hspace{1pt}\solidrule[1pt]\hspace{1pt}}}

\newcommand{\X}{\mathcal{X}}								
\newcommand{\R}{\mathbb{R}}		
\newcommand{\T}{\mathcal{T}}									
\newcommand{\V}{\mathcal{V}}								

\newcommand{\mysquare}[1]{%
	\protect\begin{tikzpicture}%
	\protect\draw[thick,color=#1] (0,0) -- (0.75ex,0) -- (0.75ex,0.75ex) -- (0,0.75ex) -- (0,0);
	\protect\end{tikzpicture}%
}
\newcommand{\mytriangle}[1]{%
	\protect\begin{tikzpicture}%
	\protect\draw[thick,color=#1] (0,0) -- (1ex,0) -- (0.5ex,0.866ex) -- (0,0);
	\protect\end{tikzpicture}%
}
\newcommand{\mycross}[1]{%
	\protect\begin{tikzpicture}%
	\protect\draw[thick,color=#1] (0,0) -- (1ex,1ex);
	\protect\draw[thick,color=#1] (0,1ex) -- (1ex,0);
	\protect\end{tikzpicture}%
}

\newcommand{\mycirc}[1]{%
	\protect\begin{tikzpicture}%
	\protect\draw[thick,color=#1] (0.5ex,0.5ex) circle (0.5ex);
	\protect\end{tikzpicture}%
}

\newtheorem{theorem}{Theorem}
\newtheorem{lemma}{Lemma}
\newtheorem{proposition}{Proposition}

\theoremstyle{definition}
\newtheorem{remark}{Observation}
\newtheorem{example}{Example}

\numberwithin{equation}{section}
\numberwithin{theorem}{section}
\numberwithin{lemma}{section}
\numberwithin{proposition}{section}
\numberwithin{definition}{section}
\numberwithin{remark}{section}
\numberwithin{example}{section}

\makeatletter
\newcommand{\subalign}[1]{%
	\vcenter{%
		\Let@ \restore@math@cr \default@tag
		\baselineskip\fontdimen10 \scriptfont\tw@
		\advance\baselineskip\fontdimen12 \scriptfont\tw@
		\lineskip\thr@@\fontdimen8 \scriptfont\thr@@
		\lineskiplimit\lineskip
		\ialign{\hfil$\m@th\scriptstyle##$&$\m@th\scriptstyle{}##$\crcr
			#1\crcr
		}%
	}
}
\newcommand*\fsize{\dimexpr\f@size pt\relax}
\makeatother        
\usepackage[sort&compress,numbers]{./natbib}
\renewcommand{\bibfont}{\small}
\author[1]{Giovanni Fantuzzi\thanks{Email address for correspondence: \href{mailto:giovanni.fantuzzi10@imperial.ac.uk}{giovanni.fantuzzi10@imperial.ac.uk}}} 
\author[2]{David Goluskin\thanks{Email address for correspondence: \href{mailto:goluskin@uvic.ca}{goluskin@uvic.ca}}}
\affil[1]{Department of Aeronautics, Imperial College London, London, SW7 2AZ, United Kingdom.}
\affil[2]{Department of Mathematics and Statistics, University of Victoria, Victoria, BC, V8P 5C2, Canada.}
\title{Bounding	extreme events in nonlinear dynamics\\using convex optimization}

\begin{document}

\maketitle

\begin{small}
\noindent\textbf{Abstract.~}We study a convex optimization framework for bounding extreme events in nonlinear dynamical systems governed by ordinary or partial differential equations (ODEs or PDEs). This framework bounds from above the largest value of an observable along trajectories that start from a chosen set and evolve over a finite or infinite time interval. The approach needs no explicit trajectories. Instead, it requires constructing suitably constrained auxiliary functions that depend on the state variables and possibly on time. Minimizing bounds over auxiliary functions is a convex problem dual to the non-convex maximization of the observable along trajectories. {This duality is strong,  meaning that auxiliary functions give arbitrarily sharp bounds, for sufficiently regular ODEs evolving over a finite time on a compact domain.} When these conditions fail, strong duality may or may not hold; both situations are illustrated by examples. We also show that near-optimal auxiliary functions can be used to construct spacetime sets that localize trajectories leading to extreme events. Finally, in the case of polynomial ODEs and observables, we describe how polynomial auxiliary functions of fixed degree can be optimized numerically using polynomial optimization. The corresponding bounds become sharp as the polynomial degree is raised if strong duality and mild compactness assumptions hold. Analytical and computational ODE examples illustrate the construction of bounds and the identification of extreme trajectories, along with some limitations. As an analytical PDE example, we bound the maximum fractional enstrophy of solutions to the Burgers equation with fractional diffusion.

\vspace*{-2ex}
\end{small}

\section{Introduction}
\label{s:intro}

Predicting the magnitudes of extreme events in deterministic dynamical systems is a fundamental problem with a wide range of applications. Examples of practical relevance include estimating the amplitudes of rogue waves in fluid or optical systems~\cite{Onorato2013}, the fastest possible mixing by incompressible fluid flows~\cite{Foures2014,Marcotte2018}, and the largest load on a structure due to dynamical forcing. In addition, extreme events relating to finite-time singularity formation are central to mathematical questions about the well-posedness and regularity of partial differential equations (PDEs). One such question is the Millennium Prize Problem concerning regularity of the three-dimensional Navier--Stokes equations~\cite{Carlson2006}, for which finite bounds on various quantities that grow transiently would imply the global existence of smooth solutions~\cite{Foias1989, Doering1995, Escauriaza2003, Doering2009}.

This work studies extreme events in dynamical systems governed by ordinary differential equations (ODEs) or PDEs. Specifically, given a scalar quantity of interest $\Phi$, we seek to bound its largest possible value along trajectories that evolve forward in time from a prescribed set of initial conditions. This maximum, denoted by $\maxphi$ and defined precisely in the next section, may be considered over all forward times or up to a finite time. Our definition of extreme events as maxima applies equally well to minima since a minimum of $\Phi$ is a maximum of $-\Phi$.

Bounding $\maxphi$ from above and from below are fundamentally different tasks. A lower bound is implied by any value of $\Phi$ on any relevant trajectory, whereas upper bounds are statements about whole classes of trajectories and require a different approach. Analytical bounds of both types appear in the literature for many systems with complicated nonlinear dynamics, but often they are far from sharp. More precise lower bounds on $\maxphi$ have sometimes been obtained using numerical integration, for instance to study extreme transient growth, optimal mixing, and transition to turbulence in fluid mechanics~\cite{Ayala2011, Ayala2014, Farazmand2017,Foures2014, Marcotte2018,Kerswell2014}. In such computations, adjoint optimization~\cite{Gunzburger2003} is used to search for an initial condition that locally maximizes $\Phi$ at a fixed terminal time, and a second level of optimization can vary the terminal time. Since both optimizations are non-convex, they give a local maximum of $\Phi$ but do not give a way to know whether it coincides with the global maximum $\maxphi$ or is strictly smaller. Thus, adjoint optimization cannot give upper bounds on $\maxphi$, even when made rigorous by interval arithmetic. {To find such an upper bound using numerical integration, one could use verified computations to find an outer approximation to the reachable set of trajectories starting from a bounded set~\cite{Cyranka2017}, and then bound $\maxphi$ from above by the global maximum of $\Phi$ on this approximating set. However, the latter is hard to compute if either $\Phi$ or the set on which it must be maximized are not convex.}

The present study describes a general framework for bounding $\maxphi$ from above that does not rely on numerical integration. This framework can be implemented analytically, computationally, or both, depending on what is tractable for the equations being studied. It falls within a broad family of methods, dating back to Lyapunov's work on nonlinear stability~\cite{Lyapunov1892}, whereby properties of dynamical systems are inferred by constructing \emph{auxiliary functions}, which depend on the system's state and possibly on time, and which satisfy suitable inequalities. Lyapunov functions \cite{Lyapunov1892,Datko1970}, which often are used to verify nonlinear stability, are one type of auxiliary functions. Other types can be used to approximate basins of attraction~\cite{Tan2006, Korda2013, Henrion2014, Valmorbida2017} and reachable sets~\cite{Magron2019,Jones2019}, estimate the effects of disturbances~\cite{Willems1972, Dashkovskiy2013, Ahmadi2016}, guarantee the avoidance of certain sets~\cite{Prajna2007, Ahmadi2017}, design nonlinear optimal controls~\cite{Lasserre2008, Henrion2008, Majumdar2014, Korda2016, Zhao2017, Korda2018}, bound infinite-time averages or stationary stochastic expectations~\cite{Chernyshenko2014, Fantuzzi2016siads, Kuntz2016, Goluskin2016lorenz, Tobasco2018, Korda2018a, Goluskin2019}, and bound extreme values over global attractors~\cite{Goluskin2018}. Some of these works refer to auxiliary functions as Lyapunov, Lyapunov-like, storage, or barrier functions, or as subsolutions to the Hamilton--Jacobi equation. Others do not use auxiliary functions explicitly but characterize nonlinear dynamics using invariant or occupation measures; the two approaches are related by Lagrangian duality and are equivalent in many cases. Furthermore, many proofs about differential equations that rely on monotone quantities can be viewed as special cases of various auxiliary function methods. For instance, as we explain in \cref{ex:fractional-burgers}, the bounds on transient growth in fluid systems proved in~\cite{Ayala2011,Ayala2014} fit within the general framework described here. Similarly, the ``background method'' introduced in~\cite{Doering1992} to bound infinite-time averages in fluid dynamics is equivalent to using quadratic auxiliary functions in a different framework~\cite{Chernyshenko2017, Goluskin2019}.

In this paper, we describe how to use auxiliary functions to bound extreme values among nonlinear ODE or PDE trajectories starting from a specified set of initial conditions. Precisely, any differentiable auxiliary function satisfying two inequalities given in \cref{s:bounds-with-afs} provides an \textit{a priori} upper bound on $\maxphi$, without any trajectories being known. {In the field of PDE analysis, these inequality conditions have been used implicitly to bound extreme events (e.g.,~\cite{Ayala2011,Ayala2014}), but the unifying framework we describe often has gone unrecognized. In the field of control theory, generalizations of our framework appear as convex relaxations of deterministic optimal control problems (e.g.,~\cite{Vinter1978,Vinter1978a,Lewis1980,Vinter1993}) and of stochastic optimal stopping problems~\cite{Cho2002}. In these works, constraints on auxiliary functions are deduced using convex duality after replacing the maximization of $\Phi$ over trajectories with a convex maximization over occupation measures. Here we derive the same constraints using elementary calculus, and we illustrate their application using numerous ODE examples and one PDE example.} 

Unlike the maximization over trajectories that defines $\maxphi$, seeking the smallest upper bound among all admissible auxiliary functions defines a convex minimization problem. In general these two optimization problems are weakly dual: the minimum is an upper bound on the maximum but may not be equal to it. In some cases they are strongly dual, meaning that the maximum over trajectories coincides with the minimum over auxiliary functions, and these functions act as Lagrange multipliers that enforce the dynamics when maximizing $\Phi$ over trajectories. In such cases there exist auxiliary functions giving arbitrarily sharp upper bounds on $\maxphi$. Strong duality holds for a large class of sufficiently regular ODEs where the maximum of $\Phi$ is taken over a finite time horizon. {This strong duality has been proved for a more general class of optimal control problems using measure theory and convex duality~\cite{Lewis1980}, and \cref{s:direct-proof-strong-duality} gives a simpler proof for our present context that shows existence of near-optimal auxiliary functions using a mollification argument similar to~\cite{Hernandez1996}.}

In many practical applications, constructing auxiliary functions that yield explicit upper bounds on $\maxphi$ is difficult regardless of whether strong duality holds. We illustrate various constructions here but do not have an approach that works universally. However, in the important case of dynamical systems governed by polynomial ODEs, polynomial auxiliary functions can be constructed using computational methods for polynomial optimization. With an infinite time horizon, this approach is applicable if the only invariant trajectories are algebraic sets, which is always true of steady states and is occasionally true of periodic orbits. With a finite time horizon, there is no such restriction. Polynomial ODEs are computationally tractable because the inequality constraints on auxiliary functions amount to nonnegativity conditions on certain polynomials. Polynomial nonnegativity is NP-hard to decide~\cite{Murty1987} but can be replaced by the stronger constraint that the polynomial is representable as a sum of squares (SOS). Optimization problems subject to SOS constraints can be reformulated as semidefinite programs (SDPs)~\cite{Nesterov2000, Lasserre2001, Parrilo2003} and solved using algorithms with polynomial-time complexity~\cite{Vandenberghe1996}. Thus, one can minimize upper bounds on $\maxphi$ for polynomial ODEs by numerically solving SOS optimization problems. Moreover, we prove that bounds computed with SOS methods becomes sharp as the degree of the polynomial auxiliary function is raised, provided that the time horizon is finite, certain compactness properties hold, and the minimization over general auxiliary functions is strongly dual to the maximization of $\Phi$ over trajectories. We illustrate the computation of very sharp bounds using SOS methods for several ODE examples, including a 16-dimensional system.

In addition to methods for bounding $\maxphi$ above, we describe a way to locate trajectories on which the observable $\Phi$ attains its maximum value of $\maxphi$. Specifically, auxiliary functions that prove sharp or nearly sharp upper bounds on $\maxphi$ can be used to define regions in state space where each such trajectory must lie prior to its extreme event. We illustrate this using an ODE for which nearly optimal polynomial auxiliary functions can be computed by SOS methods.

The rest of this paper is organized as follows. \Cref{s:bounds-with-afs} explains how auxiliary functions can be used to bound the magnitudes of extreme events in nonlinear dynamical systems. We construct bounds in several ODE examples and one PDE example; some but not all of these bounds are sharp. \Cref{s:optimal-trajectories} explains how auxiliary functions can be used to locate trajectories leading to extreme events. \Cref{s:sos-optimization} describes {how polynomial optimization can be used to construct auxiliary functions computationally} for polynomial ODEs. {Bounds computed in this way for various ODE examples appear in that section and others.} \Cref{s:extensions} extends the framework to give bounds on extreme values at particular times or integrated over time, rather than maximized over time, {giving a more direct derivation of bounding conditions that have appeared in~\cite{Vinter1978,Vinter1978a,Lewis1980,Vinter1993}.} Conclusions and open questions are offered in \cref{s:conclusion}. Appendices contain details of calculations and an alternative proof of the strong duality result that follows from~\cite{Lewis1980}.

\section{Bounds using auxiliary functions}
\label{s:bounds-with-afs}

Consider a dynamical system on a Banach space $\X$ that is governed by the differential equation
\begin{equation}
	\label{e:system}
	\dot{x} = F(t,x), \quad x(t_0) = x_0.
\end{equation}
Here, $F:\R \times \X \to \X$ is continuous and possibly nonlinear, the initial time $t_0$ and initial condition $x_0$ are given, and $\dot x$ denotes $\partial_t x$. When $\X = \R^n$, \cref{e:system} defines an $n$-dimensional system of ODEs. When $\X$ is a function space and $F$ a differential operator, \cref{e:system} defines a parabolic PDE, which may be considered in either strong or weak form~\cite{Temam1997, Robinson2001}. The trajectory of~\cref{e:system} that passes through the point $y \in \X$ at time $s$ is denoted by $x(t \given s,y)$. We assume that, for every choice of $(s,y) \in \R \times \X$, this trajectory exists uniquely on an open time interval, which can depend on both $s$ and $y$ and might be unbounded.

Suppose that $\Phi : \R \times \mathcal{X} \to \R$ is a continuous function that describes a quantity of interest for system~\cref{e:system}. Let $\Phi^*$ denote the largest value attained by $\Phi[t,x(t\given t_0, x_0)]$ among all trajectories that start from a prescribed set $X_0 \subset \X$ and evolve forward over a closed time interval $\mathcal T$ that is either finite, $\mathcal T=[t_0,T]$, or infinite, $\mathcal T=[t_0,\infty)$:
\begin{equation}
	\label{e:maxphi}
	\maxphi := \sup_{\substack{x_0 \in X_0\\[2pt]t\in\mathcal T}} \Phi\!\left[ t, x(t \given t_0, x_0) \right].
\end{equation}
We write $\maxphiT$ and $\maxphiinf$ instead of $\maxphi$ when necessary to distinguish between finite and infinite time horizons.
Our objective is to bound $\Phi^*$ from above without knowing trajectories of~\cref{e:system}.

Let $\Omega \subset \T \times \X$ be a region of spacetime in which the graphs $(t,x(t\given t_0,x_0))$ of all trajectories starting from $X_0$ remain up to the time horizon of interest. In applications one may be able to identify a set $\Omega$ that is strictly smaller than $\T \times \X$, otherwise it suffices to choose $\Omega=\T \times \X$. The maximum~\cref{e:maxphi} that we aim to bound depends only on trajectories within~$\Omega$.

To derive upper bounds on $\maxphi$ we employ auxiliary functions $V:\Omega\to\R$. In most cases we require $V$ to be differentiable along trajectories of~\cref{e:system}, so that its Lie derivative
\begin{equation}
	\mathcal{L} V(s,y) := \lim_{\varepsilon \to 0} \frac{V \left[s+\varepsilon, x(s+\varepsilon \given s,y)\right] - V(s,y)}{\varepsilon}
\end{equation}
is well defined. By design the function $\mathcal LV:\Omega\to\R$ coincides with the rate of change of $V$ along trajectories, meaning $\ddt V(t,x(t))=\mathcal LV(t,x(t))$ if $x(t)$ solves~\cref{e:system} and all derivatives exist. Crucially, an expression for $\mathcal LV$ can be derived without knowing the trajectories. In practice one differentiates $V[t,x(t\given s,y)]$ with respect to $t$ and uses the differential equation~\cref{e:system}. For example, when $\X = \R^n$ and~\cref{e:system} is a system of ODEs, the chain rule gives
\begin{equation}
	\label{e:LV-odes}
	\mathcal{L} V(t,x) =  \partial_t V(t,x) +  F(t,x)\cdot \nabla_x V(t,x).
\end{equation} 

{\Cref{ss:framework} presents inequality constraints on $V$ and $\mathcal L V$ that imply upper bounds on $\Phi^*$, as well as a convex framework for optimizing these bounds. Both can be obtained as particular cases of a general relaxation framework for optimal control problems~\cite{Vinter1978,Vinter1978a,Lewis1980}, but we give an elementary derivation.} \Cref{ss:global-local} compares bounds obtained when $\Omega=\T \times \X$, meaning that the constraints on $V$ are imposed globally in spacetime, to bounds obtained when a strictly smaller $\Omega$ containing all relevant trajectories can be found. {Finally, \cref{ss:sharpness} discusses conditions under which arbitrarily sharp upper bounds on $\Phi^*$ can be proved.}

\subsection{Bounding framework}
\label{ss:framework}

Assume that for each initial condition $x_0 \in X_0$ a trajectory $x(t\given t_0, x_0)$ exists on some open time interval where it is unique and absolutely continuous. This does not preclude trajectories that are unbounded in infinite or finite time. To bound $\maxphi$ we define a class $\V(\Omega)$ of admissible auxiliary functions as the subset of all differentiable functions, $C^1(\Omega)$, that do not increase along trajectories and bound $\Phi$ from above pointwise. Precisely, $V \in \V(\Omega)$ if and only if
\begin{subequations}
	\label{e:V-conditions}
	\begin{align}
		\label{e:cond1}
		\mathcal{L}V(t,x) &\leq 0 \quad\forall (t,x) \in \Omega,\\
		\label{e:cond2}
		\Phi(t,x) - V(t, x) &\leq 0  \quad\forall (t,x) \in \Omega.
	\end{align}
\end{subequations}
The system dynamics enter only in the derivation of $\mathcal{L}V$; conditions~(\ref{e:V-conditions}a,b) are imposed pointwise in the spacetime domain $\Omega$ and can be verified without knowing any trajectories. 
If $\Omega = \T\times \X$ we call $V$ a \emph{global} auxiliary function, otherwise it is \emph{local} on a smaller chosen $\Omega$.

We claim that
\begin{equation}
	\label{e:weak-duality}
	\maxphi \leq \adjustlimits \inf_{V \in \V(\Omega)}\sup_{x_0 \in X_0} V(t_0,x_0),
\end{equation}
with the convention that the righthand side is $+\infty$ if $\V(\Omega)$ is empty. To see that \cref{e:weak-duality} holds when $\V$ is not empty, consider fixed $V\in\V(\Omega)$ and $x_0\in X_0$. For any $t\geq t_0$ up to which the trajectory $x(t\given t_0,x_0)$ exists and is absolutely continuous, the fundamental theorem of calculus can be combined with~(\ref{e:V-conditions}a,b) to find
\begin{align}
	\label{e:inequality-sequence}
	\Phi[t, x(t \given t_0, x_0)] 
	&\leq V[t, x(t \given t_0, x_0)]\\\nonumber
	&= V(t_0,x_0) + \int_{t_0}^t \mathcal{L}V[\xi ,x(\xi \given t_0, x_0)] \dxi\\\nonumber
	&\leq V(t_0,x_0).
\end{align}
Thus, the existence of any $V \in \V(\Omega)$ implies that $\Phi[t, x(t \given t_0, x_0)]$ is bounded uniformly on $\T$ for each $x_0$. Conversely, if $\Phi$ blows up before the chosen time horizon for any $x_0\in X_0$, then no auxiliary functions exist. Maximizing both sides of \cref{e:inequality-sequence} over $t\in\T$ and $x_0\in X_0$ gives
\begin{equation}
	\label{e:weak-duality-incomplete}
	\maxphi \leq \sup_{x_0 \in X_0} V(t_0,x_0),
\end{equation}
and then minimizing over $\V(\Omega)$ gives \cref{e:weak-duality} as claimed.

The minimization problem on the righthand side of~\cref{e:weak-duality} is convex and gives a bound on the (generally non-convex) maximization problem defining $\maxphi$ in~\cref{e:maxphi}. Despite convexity of the minimization, it usually is difficult to construct an optimal or near-optimal auxiliary function, even with computer assistance. Nevertheless, any auxiliary function satisfying~(\ref{e:V-conditions}a,b) gives a rigorous upper bound on $\maxphi$ according to~\cref{e:weak-duality-incomplete}. This framework therefore can be useful for analysis, and sometimes for computation, even when the dynamics are very complicated. Analytically, one often can find a suboptimal auxiliary function that yields fairly good bounds. Computationally, for certain systems including polynomial ODEs, one can optimize $V$ over a finite-dimensional subset of $\V(\Omega)$ to obtain bounds that are very good and sometimes perfect. However, the inequality in~\cref{e:weak-duality} is strict in general, meaning that there are cases where the optimal bounds provable using conditions~(\ref{e:V-conditions}a,b) are not sharp. Local auxiliary functions can sometimes produce sharp bounds when global ones fail, although this depends on the spacetime set $\Omega$ inside which the graphs of trajectories are known to remain. This is illustrated by examples in \cref{ss:global-local}, while \cref{ss:sharpness} discusses sufficient conditions for bounds from auxiliary functions to be arbitrarily sharp. First, however, we present two examples where global auxiliary functions work well.

\Cref{ex:nonautonomous-example-sos} concerns a simple ODE where the optimal upper bound~\cref{e:weak-duality} produced by global $V$ appears to be sharp. We conclude this by constructing $V$ increasingly near to optimal, obtaining bounds that are extremely close to $\maxphi$. These $V$ are constructed computationally using polynomial optimization methods, the explanation of which is postponed until \cref{s:sos-optimization}. \Cref{ex:fractional-burgers} proves bounds for the Burgers equation with ordinary and fractional diffusion. We analytically construct $V$ giving bounds that are finite, but unlikely to be sharp. The bounds for fractional diffusion are novel, while those for ordinary diffusion show that the proof of the same result in~\cite{Ayala2011} can be seen as an instance of {the} auxiliary function framework.

\begin{example}
	\label{ex:nonautonomous-example-sos}
	\belowpdfbookmark{Example~\ref{ex:nonautonomous-example-sos}}{bookmark:nonautomonous-sos} 
	
	Consider the nonautonomous ODE system
	\begin{equation}
		\label{e:nonautomonous-system-example}
		\begin{bmatrix}
			\dot{x}_1 \\ \dot{x}_2
		\end{bmatrix} 
		= 
		\begin{bmatrix}
			x_2 t -0.1 x_1 -x_1 x_2\\ -x_1 t -x_2 +x_1^2
		\end{bmatrix}.
	\end{equation} 
	All trajectories eventually approach the origin, but various quantities can grow transiently. For example, consider the maximum of $\Phi = x_1$ over an infinite time horizon. Let the initial time be $t_0=0$ and the set of initial conditions $X_0$ contain only the point $x_0=(0,1)$. Then, $\maxphiinf$ is the largest value of $x_1$ along the trajectory with $x(0)=(0,1)$, and it is easy to find by numerical integration. Doing so gives $\maxphi\approx 0.30056373$, and this value can be used to judge the sharpness of upper bounds on $\maxphiinf$ that we produce using global auxiliary functions.
	
	The quadratic polynomial
	\begin{equation}
		\label{e:nonautomonous-system-example-V}
		V(t,x) = \tfrac12 \left( 1 + x_1^2 + x_2^2\right)
	\end{equation}
	is an admissible global auxiliary function, meaning that it satisfies the inequalities~(\ref{e:V-conditions}a,b) on $\Omega=[0,\infty)\times\R^2$. For this $V$ and the chosen $X_0$ and $t_0$, the bound \cref{e:weak-duality-incomplete} yields
	\begin{equation}
		\maxphiinf \leq V(0,x_0) = 1.
	\end{equation}
	This is the best bound that can be proved using global quadratic $V$, as shown in \cref{app:nonautomonous-system-example-optimality}, but optimizing polynomial $V$ of higher degree produces better results. For instance, the best global quartic $V$ that can be constructed using polynomial optimization is
	\begin{multline}
		V(t,x)=
		0.2353
		+0.7731\,x_1^2
		+0.1666\,x_1 x_2
		+0.4589\,x_2^2
		+0.5416\,x_1^3
		+0.05008\,t x_1^2\\
		+0.1616\,t x_1 x_2
		+0.2505\,t x_2^2
		-0.1058\,x_1^2 x_2
		+0.1730\,x_1 x_2^2
		-0.5766\,x_2^3\\
		+0.2962\,x_1^4
		+0.1888\,t^2 x_1^2
		+0.1888\,t^2 x_2^2
		+0.5923\,x_1^2 x_2^2
		+0.2962\,x_2^4,
	\end{multline}
	where numerical coefficients have been rounded. The bound on $\maxphiinf$ that follows from the above $V$ is reported in \cref{table:results-nonautonomous-example}, along with bounds that follow from computationally optimized $V$ of polynomial degrees 6, 8, and 10 (omitted for brevity). The bounds improve as the degree of $V$ is raised, and the optimal degree-8 bound is sharp up to nine significant figures. The numerical approach used for such computations is described in \cref{s:sos-optimization}.
	
	\begin{table}[t]
		\caption{Upper bounds on $\maxphiinf$ for \cref{ex:nonautonomous-example-sos}, computed using polynomial optimization with $V$ of various polynomial degrees. For the single initial condition $x_0=(0,1)$, numerical integration gives $\maxphi\approx0.30056373$ for all time horizons larger than $T=1.6635$, which agrees with the degree-8 bound to the tabulated precision. For the set $X_0$ of initial conditions on the shifted unit circle with center $(-\tfrac34,0)$, nonlinear optimization of the initial angular coordinate yields $\maxphiinf\approx0.49313719$, which agrees with the degree-10 bound to the tabulated precision.}
		\label{table:results-nonautonomous-example}
		\centering
		\small
		\begin{tabular}{cc ccc}
			\toprule
			&& \multicolumn{3}{c}{Upper bounds} 	 \\[2pt]
			\cline{3-5}\\[-8pt]
			$\deg(V)$	&& $X_0=\{(0,1)\}$ && $X_0$ circle \\[2pt]
			\hline
			2   	    	&& 1\phantom{.00000000}	&& 1.75\phantom{000000} \\
			4		&& 0.41381042			&& 0.80537235 \\
			6		&& 0.30056854			&& 0.49808038 \\
			8		&& 0.30056373			&& 0.49313760 \\ 
			10		&& ''					&& 0.49313719 \\
			\bottomrule
		\end{tabular}
	\end{table}
	
	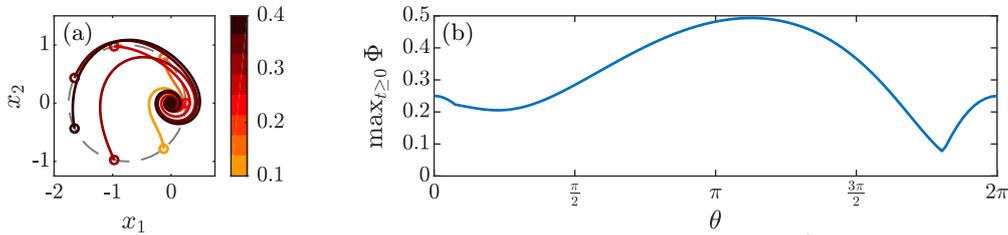
\begin{figure}
		\centering
		\includegraphics[scale=1]{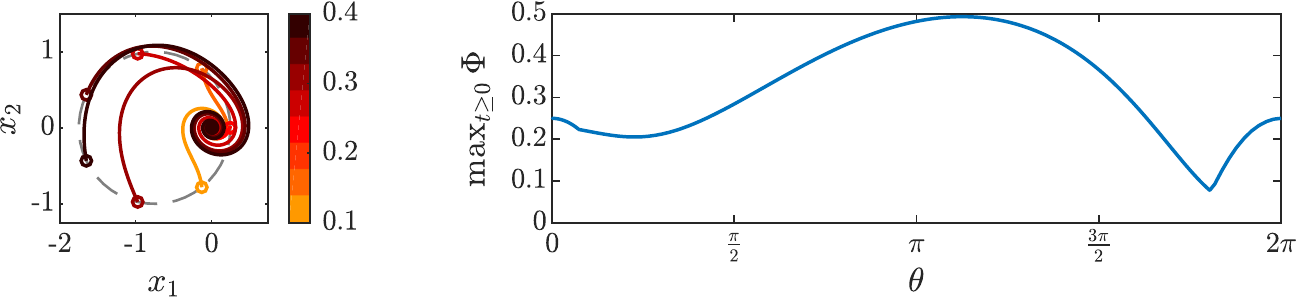}
		\begin{tikzpicture}[overlay]
		\node at (-12.4,2.65) {\footnotesize(a)};
		\node at (-7.4,2.65) {\footnotesize(b)};
		\end{tikzpicture}
		\vspace{-2ex}
		\caption{(a) Sample trajectories starting from the circle with center $(-\tfrac34,0)$ and unit radius ({\color{matlabgray}\dashedrule}). The initial conditions are marked with a circle, while the color scale reflects the maximum value of $\Phi$ along each trajectory. (b) Numerical approximation to the maximum of $\Phi$ along single trajectories with initial condition on the shifted unit circle $(\cos\theta-\tfrac34,\sin\theta)$ as a function of the angular coordinate~$\theta$.}
		\label{f:nonautonomous-figure}
	\end{figure}
	
	Unlike searching among particular trajectories, bounding $\maxphi$ from above is not more difficult when the set $X_0$ of initial conditions is larger than a single point. For example, consider initial conditions on the shifted unit circle centered at $(-\tfrac34,0)$,
	\begin{equation}
		X_0 
		= \left\{(x_1,x_2):\,\left(x_1+\tfrac34 \right)^2+x_2^2 = 1\right\} 
		= \Big\{\left(\cos \theta - \tfrac34,\sin\theta \right):\;\theta \in [0,2\pi)\Big\}.
	\end{equation}
	Sample trajectories and the variation of $\max_{t\ge0}\Phi$ with the angular position $\theta$ in $X_0$ are shown in \cref{f:nonautonomous-figure}. Finding the trajectory that attains $\maxphi$ requires numerical integration, combined with nonlinear optimization over initial conditions in $X_0$. Starting MATLAB's optimizer \texttt{fmincon} from initial guesses with angular coordinate $\theta=\tfrac{3\pi}{4}$ and $\theta=\tfrac{\pi}{10}$ yields locally optimal initial conditions of $\theta\approx1.125\pi$ and $\theta=2\pi$, which lead to $\Phi$ values of 0.49313719 and 0.25, respectively. \Cref{f:nonautonomous-figure}(b) confirms that the former initial condition is globally optimal, meaning $\maxphi\approx0.49313719$.
	On the other hand, polynomial auxiliary functions can be optimized by the methods of \cref{s:sos-optimization} using exactly the same algorithms as when $X_0$ contains a single point. For initial conditions on the shifted unit circle $X_0$, \cref{table:results-nonautonomous-example} lists upper bounds on $\maxphi$ implied by numerically optimized polynomial $V$ of degrees up to 10. We omit the computed $V$ for brevity. The optimal degree-10 $V$ gives a bound that is sharp to eight significant figures.	
	\markendexample\end{example}

\begin{example}
	\label{ex:fractional-burgers}
	\belowpdfbookmark{Example~\ref{ex:fractional-burgers}}{bookmark:fractional-burgers}
	
	To illustrate the analytical use of global auxiliary functions for PDEs, we consider mean-zero period-1 solutions $u(t,x)$ of the Burgers equation with fractional diffusion,
	\begin{equation}
		\label{e:fractional-burgers}
		\begin{gathered}
			\dot u = - u u_x - (-\Delta)^\alpha u, \\
			u(0,x) = u_0(x),\quad
			u(t,x+1) = u(t,x),\quad
			\int_0^1 u(t,x) \dx = 0.
		\end{gathered}
	\end{equation}
	Following standard PDE notation, in this example the state variable in $\X$ is denoted by $u(t,\cdot)$, whereas $x\in[0,1]$ is the spatial variable. Discussion of this equation and a definition of the fractional Laplacian $(-\Delta)^\alpha$ can be found in~\cite{Yun2018}. Ordinary diffusion is recovered when $\alpha=1$. For each $\alpha\in(\tfrac12,1]$, solutions exist and remain bounded when the Banach space $\X$ in which solutions evolve is the Sobolev space $H^s$ with $s>\tfrac32-2\alpha$~\cite{Kiselev2008}. 	Let us consider a quantity that is called fractional enstrophy in~\cite{Yun2018},
	\begin{equation}
		\Phi(u) := \frac12 \int_0^1 \left[(-\Delta)^{\frac{\alpha}{2}} u\right]^2 \dx.
	\end{equation}
	We aim to bound $\maxphi_\infty$ among trajectories whose initial conditions $u_0$ have a specified value $\Phi_0$ of fractional enstrophy, so the set of initial conditions is
	\begin{equation}
		X_0=\left\{ u\in\X :\,\Phi(u)=\Phi_0 \right\}.
	\end{equation}
	
	Here we prove $\Phi_0$-dependent upper bounds on $\maxphiinf$ for $\alpha\in(\tfrac34,1]$. Such bounds have been reported for ordinary diffusion ($\alpha=1$) \cite{Ayala2011} but not for $\alpha<1$. We employ global auxiliary functions of the form
	\begin{equation}
		V(u) = \left[ \Phi(u)^\beta + C \|u\|_2^2 \right]^{1/\beta},
		\label{e:burgers-V-ansatz}
	\end{equation}
	where $\|u\|_2^2 = \int_0^1 u^2 \dx$ and the constants $\beta,C>0$ are to be chosen. This ansatz is guided by the realization that the analysis of the $\alpha=1$ case~\cite{Ayala2011} is equivalent to {the} auxiliary function framework with $\beta=1/3$ in~\cref{e:burgers-V-ansatz}. 
	
	To be an admissible auxiliary function, $V$ must satisfy~(\ref{e:V-conditions}a,b). The inequality $V(u) \geq \Phi(u)$ holds for every positive $C$, while the inequality $\mathcal{L}V(u)\le 0$ constrains $\beta$ and $C$.
	To derive an expression for $\mathcal{L}V(u)$ we first note that differentiating along trajectories of \cref{e:fractional-burgers} and integrating by parts gives
	\begin{subequations}
		\label{e:fractional-burgers-KE}
		\begin{gather}
			\ddt \|u(t,\cdot)\|_2^2 = - 4  \Phi[u(t,\cdot)],\\[1ex]
			\ddt \Phi[u(t,\cdot)] =R[u(t,\cdot)] := - \int_0^1 [(-\Delta)^\alpha u]^2 \dx - \int_0^1 u u_x (-\Delta)^\alpha u \dx.
		\end{gather}
	\end{subequations}
	Differentiating $V[u(t,\cdot)]$ in time thus gives
	\begin{equation}
		\label{e:fractional-burgers-LV}
		\mathcal{L}V(u) = \frac{1}{\beta}\left[ \Phi(u)^\beta + C \|u\|_2^2 \right]^{ \frac{1}{\beta} -1 } \left[ \beta \Phi(u)^{\beta-1} R(u) -  4 C \Phi(u) \right].
	\end{equation}
	The sign of $\mathcal L V$ is that of the expression in the rightmost brackets, so an estimate for $R(u)$ is needed. Theorem 2.2 in~\cite{Yun2018} provides $R(u) \leq \sigma_\alpha \Phi(u)^{\gamma_\alpha}$, with $\gamma_\alpha=\tfrac{8\alpha-3}{6\alpha-3}$ and explicit prefactors $\sigma_\alpha$ that blow up as $\alpha\to\tfrac34^+$. By fixing $\beta=2-\gamma_\alpha$ and $C = (2-\gamma_\alpha) \sigma_\alpha/4$, we guarantee that \cref{e:fractional-burgers-LV} is nonpositive. Thus, $V$ is a global auxiliary function yielding the bound
	\begin{equation}
		\maxphiinf
		\leq \sup_{u_0 \in X_0} \left[ \Phi_0^{2-\gamma_\alpha} + \frac{(2-\gamma_\alpha) \sigma_\alpha}{4} \, \|u_0\|_2^2 \right]^{ \frac{1}{2-\gamma_\alpha} }
	\end{equation}
	according to~\cref{e:weak-duality-incomplete}. Finally, the righthand maximization over $u_0$ can be carried out analytically by calculus of variations to bound $\maxphiinf$ in terms of only the initial fractional enstrophy $\Phi_0$,
	\begin{equation}
		\label{e:fractional-burgers-bound}
		\maxphiinf \leq \left[ \Phi_0^{2-\gamma_\alpha} + \frac{(2-\gamma_\alpha) \sigma_\alpha}{2 (2\pi)^{2\alpha}} \, \Phi_0 \right]^{ \frac{1}{2-\gamma_\alpha} }.
	\end{equation}

	The bound~\cref{e:fractional-burgers-bound} is finite for every $\alpha \in (\frac34,1]$. The coefficient on $\Phi_0$ is bounded uniformly for $\alpha$ in this range, but the  exponent $\tfrac{1}{2-\gamma_\alpha}$  blows up as $\alpha\to\tfrac34^+$. When $\alpha=1$ we can replace $\sigma_\alpha$ with a smaller prefactor from~\cite{Lu2008} to find
	\begin{equation}
		\label{e:ordinary-burgers-bound}
		\maxphiinf \leq \left( \Phi_0^{1/3} + 2^{-10/3}\pi^{-8/3} \, \Phi_0 \right)^3.
	\end{equation}
	The above estimate is identical to the result of~\cite{Ayala2011},\footnote{Expression (5) in~\cite{Ayala2011} is claimed to hold with $\mathcal E$ being identical to our $\Phi(u)$, but in fact it holds with $\mathcal E=2\Phi(u)$ because their derivation uses estimate (3.7) from~\cite{Lu2008}. With this correction, and with $L=1$ and $\nu=1$, the expression in~\cite{Ayala2011} agrees with our bound~\cref{e:ordinary-burgers-bound}.}
	and their argument is equivalent to ours in that it implicitly relies on our $V$ being nonincreasing along trajectories. Similarly, in~\cite{Ayala2014} the same authors bound a quantity called palinstrophy in the two-dimensional Navier--Stokes equations, and that proof can be seen as using (in their notation) the global auxiliary function $V(u) = \left[ \mathcal P(u)^{1/2} + (4\pi\nu^2)^{-2}\mathcal K(u)^{1/2} \mathcal E(u) \right]^2$.
	
	The bound~\cref{e:fractional-burgers-bound} is unlikely to be sharp. For $\alpha=1$ it scales like $\maxphiinf\leq \mathcal{O}\big(\Phi_0^3\big)$ when $\Phi_0\gg1$, whereas numerical and asymptotic evidence suggests that $\maxphiinf = \mathcal{O}\big(\Phi_0^{3/2}\big)$~\cite{Ayala2011,Pelinovsky2012}. It is an open question whether going beyond the $V$ ansatz \cref{e:burgers-V-ansatz} can produce sharper analytical bounds, and whether the optimal bound \cref{e:weak-duality} that can be proved using global auxiliary functions would be sharp in this case.	
	\markendexample\end{example}

\subsection{Global versus local auxiliary functions}
\label{ss:global-local}

In various cases, such as \cref{ex:nonautonomous-example-sos} above, global auxiliary functions can produce arbitrarily sharp upper bounds on $\maxphi$. Other times they cannot. In \cref{ex:ex-1d-unbounded-trajectories} below, global auxiliary functions give bounds that are finite but not sharp. In \cref{ex:ex-infeasible-bounded-phi}, no global auxiliary functions exist. Sharp bounds can be recovered in both examples by using local auxiliary functions, meaning that we enforce constraints~(\ref{e:V-conditions}a,b) only on a subset $\Omega \subsetneq \T\times \X$ of spacetime that contains all trajectories of interest.

There are various ways to determine that trajectories starting from the initial set $X_0$ remain in a spacetime set $\Omega$ during the time interval $\T$. One option is to choose a function $\Psi(t,x)$ and use global auxiliary functions to show that $\Psi^*\le B$ for initial conditions in $X_0$. This implies that trajectories starting from $X_0$ remain in the set
\begin{equation}
	\Omega := \lbrace (t,x)\in\T\times\X:\,\Psi(t,x)\leq B \rbrace.
\end{equation}
Any $\Psi$ that can be bounded using global auxiliary functions can be used, including $\Psi=\Phi$, and $\Omega$ can be refined by considering more than one $\Psi$. Another way to show that trajectories never exit a prescribed set $\Omega$ is to construct a barrier function that is nonpositive on $\{t_0\}\times X_0$, positive outside $\Omega$, and whose zero level set cannot be crossed by trajectories. Barrier functions can be constructed analytically in some cases, and computationally for ODEs with polynomial righthand sides; see~\cite{Prajna2007,Ahmadi2017} and references therein. Finally, in the polynomial ODE case the computational methods of \cite{Henrion2014} can produce a spacetime set $\Omega=\T \times X$, where $X \subsetneq \X$ is an outer approximation for the evolution of the initial set $X_0$ over the time interval $\T$. The next two examples demonstrate the differences between global and local auxiliary functions for a simple ODE where a suitable choice of $\Omega$ is apparent.

\begin{example}
	\label{ex:ex-1d-unbounded-trajectories}
	\belowpdfbookmark{Example~\ref{ex:ex-1d-unbounded-trajectories}}{bookmark:1d-unbounded-trajectories}
	
	Consider the autonomous one-dimensional ODE
	\begin{equation}
		\label{e:xdot=x2}
		\dot{x} = x^2, \qquad x(0)=x_0.
	\end{equation}
	Trajectories $x(t) = x_0/(1-x_0 t)$ with nonzero initial conditions grow monotonically. If $x_0<0$, then $x(t)\to0$ as $t\to\infty$; if $x_0>0$, then $x(t)$ blows up at the critical time $t=1/x_0$. Suppose the set of initial conditions $X_0$ includes only a single point $x_0$, the time interval is $\T=[0,\infty)$, and the quantity to be bounded is
	\begin{equation}
		\label{e:phi-ode-blowup}
		\Phi(x) = \frac{4x}{1+4x^2}.
	\end{equation}
	Since $|\Phi(x)|\le1$ uniformly, $\maxphiinf$ is finite for each $x_0$ despite the blowup of trajectories starting from positive initial conditions. Explicit solutions give
	\begin{equation}
		\label{e:maxphi-example-unbounded-trajectories}
		\maxphiinf = \begin{cases}
			0, & \phantom{0< }\,\,x_0 \leq 0,\\
			1, & 0 < x_0 \leq\tfrac12,\\
			\displaystyle\frac{4x_0}{1+4 x_0^2}, &\phantom{0<}\,\,x_0>\tfrac12.
		\end{cases}
	\end{equation}
	
	Here $X_0$ contains only one initial condition, so the optimal bound~\cref{e:weak-duality} simplifies to
	\begin{equation}
		\label{e:weak-duality-blowup-example}
		\maxphiinf \leq \inf_{V \in \V(\Omega)} V(0,x_0).
	\end{equation}
	The constant function $V\equiv1$ belongs to $\V$ for each $x_0$ and implies the trivial bound $\maxphiinf \leq 1$, which is sharp for $x_0\in(0,1/2]$. For all other $x_0\neq0$ there exist different $V$ providing sharp bounds on $\maxphiinf$, regardless of whether the domain $\Omega$ of auxiliary functions is global or local. This is shown in \cref{app:sharp-bounds-ex-x2}. At the semistable point $x_0=0$, however, sharp bounds are possible only with local auxiliary functions on certain $\Omega$.
	
	In the $x_0=0$ case, the resulting trajectory is simply $x(t)\equiv0$. Thus it suffices to enforce the auxiliary function constraints (\ref{e:V-conditions}a,b) locally on $\Omega=[0,\infty) \times \{0\}$. On this $\Omega$, the constant function $V\equiv0$ is a local auxiliary function giving the sharp bound $\maxphi\le0$. In fact, the same is true with $\Omega = [0,\infty) \times X$ for any $X$ with $0\in X \subseteq (-\infty,0]$. On the other hand, if the chosen set $X$ contains any open neighborhood of 0, then sharp bounds are not possible. This is true in particular for global auxiliary functions, which must satisfy constraints (\ref{e:V-conditions}a,b) on $\Omega=[0,\infty)\times\R$. The righthand minimum in~\cref{e:weak-duality-blowup-example} over global auxiliary functions is attained by the constant function $V=1$. No better bound is possible with global $V$ because they must satisfy $V(0,0) \geq 1$. To prove this, recall that every $V(t,x)$ is continuous by definition. Thus for any $\delta > 0$ there exists $y>0$ such that $V(0,0) \geq V(0,y) - \delta$. The trajectory of~\cref{e:xdot=x2} with initial condition $x(0)=y$ blows up in finite time and must therefore pass through $x=\frac12$ at some time $t^*$. Condition~\cref{e:cond2} requires that $V(t^*,\frac12) \geq \Phi(\frac12) = 1$, while~\cref{e:cond1} implies that $V$ decays along trajectories, so
	\begin{equation}
		V(0,0) \geq V(0,y) - \delta \geq V(t^*,\tfrac12) - \delta \geq 1-\delta
	\end{equation} 
	for every $\delta>0$. Thus $V(0,0) \geq 1$, so when $x_0=0$ the righthand minimum over global $V$ in~\cref{e:weak-duality-blowup-example} is indeed attained by $V\equiv1$. Local auxiliary functions can prove better bounds, but a similar argument shows that the sharp bound $\maxphi\le0$ for $X_0=\{0\}$ is possible only if $0\in X \subseteq (-\infty,0]$. That is, the upper limit of $X$ must coincide with the boundary of the basin of attraction of the semistable point at 0. In more complicated systems it may not be possible to locate $X$ so precisely. In such cases, if global auxiliary functions do not give sharp bounds, local ones might not either, at least for spacetime sets $\Omega$ that one can identify in practice.	
	\markendexample\end{example}

\begin{example}
	\label{ex:ex-infeasible-bounded-phi}
	\belowpdfbookmark{Example~\ref{ex:ex-infeasible-bounded-phi}}{bookmark:infeasible-bounded-phi}
	
	In some cases, global auxiliary functions can fail to exist even if $\maxphi$ is finite. Again consider the ODE \cref{e:xdot=x2} from \cref{ex:ex-1d-unbounded-trajectories} with $\T=[0,\infty)$ and a single initial condition $X_0=\{x_0\}$, but now consider the quantity
	\begin{equation}
		\label{e:phi-no-global-V}
		\Phi(t,x) = x^2 {\rm e}^x.
	\end{equation}
	Recalling that $x(t)$ approaches zero if $x_0\le0$ and blows up otherwise, we find
	\begin{equation}
		\label{e:maxphi-example-x^2-exponential}
		\maxphiinf = \begin{cases}
			4 \, {\rm e}^{-2}, &\phantom{-2< }\,\,x_0\leq -2,\\
			x_0^2 \, {\rm e}^{x_0}, &-2<x_0\leq 0, \\
			\infty, & \phantom{-2<}\,\, x_0 > 0.
		\end{cases}
	\end{equation}
	For auxiliary functions satisfying~(\ref{e:V-conditions}a,b) globally on $\Omega=[0,\infty)\times\R$, $\V(\Omega)$ must be empty when $x_0>0$ since $\maxphiinf=\infty$. However, $\V(\Omega)$ is empty also when $x_0\le0$, despite $\maxphiinf$ being finite. This is because any global $V$ satisfying~(\ref{e:V-conditions}a,b) must be nonincreasing for trajectories starting at all $y\in\R$, not only for initial conditions in the set of interest $X_0$. In particular,
	\begin{equation}
		\label{e:empty-V-contradiction}
		V(0,y) \geq V\!\left[ t, x(t;0,y) \right] \geq \Phi\!\left[ t, x(t;0,y) \right] =  x(t;0,y)^2 \,{\rm e}^{x(t;0,y)}
	\end{equation}
	for all $y\in\R$ and all $t\ge0$, where the second inequality follows from~\cref{e:cond2}. No $V$ that is continuous on $[0,\infty)\times \R$ can satisfy~\cref{e:empty-V-contradiction} because, for each $y>0$, the rightmost expression becomes infinite as $t$ approaches the blowup time $1/x_0$. Thus, $\V(\Omega)$ is empty.
	
	Sharp bounds on finite $\maxphi$ become possible with local rather than global auxiliary functions, much as in \cref{ex:ex-1d-unbounded-trajectories}. Since $\maxphi$ is finite only when $X_0\subseteq(-\infty,0]$, and trajectories starting from any such $X_0$ stay within $X=(-\infty,0]$, conditions (\ref{e:V-conditions}a,b) can be enforced locally on $\Omega =[0,\infty) \times X$. As in \cref{ex:ex-1d-unbounded-trajectories}, it is crucial that $X$ contains no points outside the basin of the semistable equilibrium at the origin. A local $V$ giving sharp bounds is
	\begin{equation}
		V(t,x) = \begin{cases}
			4 \, {\rm e}^{-2}, &x\leq -2,\\
			x^2 \, {\rm e}^{x}, &x>-2.
		\end{cases}
	\end{equation}
	At each $x_0\le 0$ this $V$ is equal to the value~\cref{e:maxphi-example-x^2-exponential} of $\maxphiinf$ for the single trajectory starting at $x_0$. Thus, this $V$ gives a sharp bound on $\maxphiinf$ for every possible initial set $X_0\subseteq(-\infty,0]$.
	\markendexample\end{example}

\subsection{Sharpness of optimal bounds}
\label{ss:sharpness}
The best bounds on $\maxphi$ provable using auxiliary functions are often but not always sharp. 
\cref{ex:ex-1d-unbounded-trajectories,ex:ex-infeasible-bounded-phi} above show that the upper bound~\cref{e:weak-duality} can be strict, at least for infinite time horizons and global auxiliary functions. {For finite time horizons and local auxiliary functions, on the other hand, arguments in~\cite{Lewis1980} prove that \cref{e:weak-duality} is an equality provided trajectories remain in a compact set over the finite time interval of interest. \Cref{ss:finite-time-horizon} states this result and gives an explicit counterexample for infinite time horizons. \Cref{ss:discontinuous-afs} explains why sharp bounds are always possible if one allows $V$ to be discontinuous, a fact which is useful for theory but not for explicitly bounding quantities in particular systems.}

\subsubsection{Sharp bounds for ODEs with finite time horizon}
\label{ss:finite-time-horizon}
Local auxiliary functions can produce arbitrarily sharp bounds on $\maxphiT$ with finite time horizon $T$ for well posed ODEs, provided the initial set $X_0$ is compact and trajectories that start from it remain inside a compact set $X$ up to time $T$. {Precisely, Theorem 2.1 and equation (5.3) in~\cite{Lewis1980} imply the following result.}

\begin{theorem}[\cite{Lewis1980}]
	\label{th:strong-duality}
	Let $\dot{x} = F(t,x)$ be an ODE with $F$ locally Lipschitz in both arguments. Given $\Phi:\R \times \R^n \to \R$ continuous, an initial time $t_0$, a finite time interval $\T = [t_0,T]$, and a compact set of initial conditions $X_0$, define $\maxphiT$ as in~\cref{e:maxphi}. Assume that:
	\begin{enumerate}[({A}.1)]
		\item All trajectories starting from $X_0$ at time $t_0$ remain in a compact set $X$ for $t \in \T$;
		\item There exist a time $t_1 > T$ and a bounded open neighborhood $Y$ of $X$ such that, for all initial points $(s,y) \in [t_0,t_1] \times Y$, a unique trajectory $x(t \given s,y)$ exists for all $t \in [s,t_1]$.
	\end{enumerate}
	Then, letting $\V(\Omega)$ denote the set of differentiable auxiliary functions that satisfy~(\ref{e:V-conditions}a,b) on the compact set $\Omega := \T \times X$,
	\begin{equation}
		\maxphiT =  \adjustlimits \inf_{V \in \V(\Omega)}\sup_{x_0 \in X_0} V(t_0,x_0).
		\label{e:strong-duality}
	\end{equation}
\end{theorem}

{In \cref{s:direct-proof-strong-duality} we give an alternative proof of this theorem that uses mollification to construct near-optimal $V$. This construction does not yield explicit bounds on $\maxphi_T$ for particular ODEs because it invokes trajectories, which generally are not known.} Both the original proof in~\cite{Lewis1980} and our proof rely on assumptions (A.1) and (A.2) to ensure that trajectories starting in a neighborhood of $X$ remain bounded past the time horizon $T$ and are regular in the sense that the map $(s,y) \mapsto x(t \given s,y)$ is locally Lipschitz on $[t_0,t_1] \times Y$. Regularity over a spacetime set slightly larger than $\Omega$ is used to construct smooth uniform approximations to certain functions on $\Omega$ via mollification. However, the assumptions are not necessary for the equality~\cref{e:strong-duality} to hold. For instance, the example in \cref{app:sharp-bounds-ex-x2} violates assumption (A.1) when $x_0 > 0 $ and $T=1/x_0$, yet the $V$ in~\cref{e:V-example-blowup-solutions} implies sharp bounds on $\maxphiT$.

It is an open challenge to weaken the assumptions of \cref{th:strong-duality}. {With infinite time horizons, for instance, auxiliary functions give sharp bounds in some examples but not others. Sharp bounds for an infinite time horizon are illustrated in \cref{app:sharp-bounds-ex-x2}. In the next example, on the other hand, there exists a set $X$ such that infinite-time analogues of assumptions (A.1) and (A.2) hold, yet differentiable local auxiliary functions cannot give sharp bounds on~$\maxphiinf$.}

\begin{example}
	\label{ex:strong-duality-failure}
	\belowpdfbookmark{Example~\ref{ex:strong-duality-failure}}{bookmark:strong-duality-failure}
	
	Consider the one-dimensional ODE
	\begin{equation}
		\label{e:xdot=x2-x3}
		\dot{x}=x^2-x^3,
	\end{equation}
	which has two equilibria: the semistable point $x_s = 0$ and the attractor $x_a = 1$. Although no explicit analytical solution is available, trajectories exist for all times. As $t\to\infty$, they approach $x_s$ if $x_0\le0$ and approach $x_a$ if $x_0>0$. We let
	\begin{equation}
		\Phi(x)=4x(1-x)
	\end{equation}
	and seek upper bounds on $\maxphiinf$ for initial conditions in the set $X_0=[-1,0]$. All trajectories starting in $X_0$ approach $x_s$ from below, so
	\begin{equation}
		\maxphiinf = \sup_{\substack{x_0 \in X_0\\[2pt]t\in[t_0,\infty)}}\Phi[x(t;x_0)] = 0.
	\end{equation}
	Trajectories with initial conditions in $X_0=[-1,0]$ remain there, so the smallest $X$ we could choose is $X=X_0$. With this choice, $V\equiv 0$ gives a sharp upper bound. However, suppose we choose  $X = [-1,1]$, which is the smallest connected set that is globally attracting and contains $X_0$. For this $X$, assumptions analogous to (A.1) and (A.2) in \cref{th:strong-duality} hold on the infinite time interval $[0,\infty)$, yet any upper bound on $\maxphiinf=0$ provable with differentiable local $V$ cannot be smaller than 1. Indeed, any such $V$ must be continuous at $(t,x)=(0,0)$ and arguing as in \cref{ex:ex-1d-unbounded-trajectories} shows that $V(0,0) \geq 1$, so any $V$ subject to~(\ref{e:V-conditions}a,b) satisfies
	\begin{equation}
		\max_{x \in [-1,0]} V(0,x) \geq 1.
	\end{equation}
	Thus, with $X=[-1,1]$, any bound implied by~\cref{e:weak-duality} is no smaller than 1 as claimed above.
	\markendexample\end{example}

The inability of differentiable auxiliary functions to produce sharp bounds in \cref{ex:ex-1d-unbounded-trajectories,ex:strong-duality-failure} is due to the map $x_0 \mapsto x(t \given 0,x_0)$ from initial conditions to trajectories not being locally Lipschitz near the saddle point $x_s=0$. Because the time horizon is infinite, a fixed distance from $x_s$ is eventually reached by trajectories starting arbitrarily close to $x_s$. This does not happen when the time horizon is finite. We cannot say whether the strong duality result of \cref{th:strong-duality} applies with an infinite time horizon when the map $x_0 \mapsto x(t \given 0,x_0)$ is Lipschitz; {both the original proof in~\cite{Lewis1980} and our alternative in \cref{s:direct-proof-strong-duality} rely on the time interval $\T$ being compact.}

\subsubsection{Nondifferentiable auxiliary functions}
\label{ss:discontinuous-afs}

One way to guarantee that optimization over $V$ gives sharp bounds on $\maxphi$, regardless of whether the time horizon is finite or infinite, is to weaken the local sufficient condition~(\ref{e:V-conditions}a,b) by removing the requirement that $V$ is differentiable. Since the Lie derivative $\mathcal L V$ may not be defined in this case, condition~\cref{e:cond1} must be replaced with the direct constraint that $V$ does not increase along trajectories,
\begin{equation}
	\label{e:cond1-discontinuous}
	V[s+\tau,x(s+\tau \given s, y)] \leq V(s,y)  \quad \forall \tau \geq 0 \text{ and } (s,y) \in \Omega.
\end{equation}
Slight modification of the argument leading to~\cref{e:weak-duality-incomplete} then proves
\begin{equation}
	\label{e:weak-duality-discontinuous}
	\maxphiinf \leq \min_{\subalign{V:\,&\cref{e:cond2},\\&\cref{e:cond1-discontinuous}}} \, \sup_{x_0 \in X_0} V(t_0,x_0).
\end{equation}
Condition~\cref{e:cond1-discontinuous} cannot be checked when trajectories are not known exactly.\footnote{For systems with discrete-time dynamics, on the other hand, discontinuous $V$ may be practically useful. This work focuses on continuous-time dynamics, but {the} convex bounding framework {of \cref{ss:framework}} readily extends to maps $x_{n+1} = F(n,x_{n})$ when the continuous-time decay condition~\cref{e:cond1} is replaced by the discrete version of~\cref{e:cond1-discontinuous}, namely that $V[n+1,F(n,x_{n})] \leq V(n,x_n)$ for all $n \in \mathbb{N}$ and $x_n \in \X$. This can be checked directly without knowing trajectories. In addition, the computational methods described in \cref{s:sos-optimization} can be applied with minor modifications to finite-dimensional polynomial maps.} Differentiability of $V$ therefore is crucial to find explicit bounds for particular systems because the Lie derivative $\mathcal L V$ gives a way to check that $V$ is nonincreasing without knowing trajectories.

For theoretical purposes, on the other hand, nondifferentiable $V$ are useful because
\begin{equation}
	\label{e:value-function}
	V^*(s,y) := \sup_{t \geq s} \Phi[t, x(t \given s, y)]
\end{equation}
is optimal and attains equality in~\cref{e:weak-duality-discontinuous}, meaning
\begin{equation}
	\label{e:strong-duality-discontinuous}
	\maxphiinf = \min_{\subalign{V:\,&\cref{e:cond2},\\&\cref{e:cond1-discontinuous}}} \, \sup_{x_0 \in X_0} V(t_0,x_0) = \sup_{x_0\in X_0}V^*(t_0,x_0).
\end{equation}
This $V^*$ is discontinuous in general because of the maximization over time. It follows directly from the definition of $\maxphiinf$ that $V^*$ satisfies~\cref{e:cond2} globally and gives a sharp bound when substituted into~\cref{e:strong-duality-discontinuous}. To see that~\cref{e:cond1-discontinuous} holds, observe that the trajectory starting from $y$ at time $s$ is the same as that starting from $x(s+\tau\given s,y)$ at time $s+\tau$. Then, since $\tau \geq 0$,
\begin{align}
	V^*[s+\tau,x(s+\tau \given s, y)]
	&= \sup_{t  \geq s+\tau} \Phi\{t, x[t \given s+\tau, x(s+\tau \given s, y) ]\} \\
	&= \sup_{t  \geq s+\tau} \Phi[t, x(t \given s, y)]  \notag \\
	&\leq \sup_{t  \geq s} \Phi[t, x(t \given s, y)] \notag \\
	&= V^*(s,y). \notag
\end{align}
\Cref{ex:dicontinuous-af} below gives $V^*$ in a case where trajectories are known.


\begin{example}
	\label{ex:dicontinuous-af}
	\belowpdfbookmark{Example~\ref{ex:dicontinuous-af}}{bookmark:dicontinuous-af}
	
	Recall \cref{ex:ex-1d-unbounded-trajectories}, which shows that differentiable global auxiliary functions cannot give sharp bounds for the ODE~\cref{e:xdot=x2} with $\Phi$ as in~\cref{e:phi-ode-blowup} and the single initial condition $X_0=\{0\}$. For the auxiliary function
	\begin{equation}
		V(t,x) = \begin{cases}
			0, & \phantom{0< }\,\,x \leq 0,\\
			1, & 0 < x \leq\tfrac12,\\
			\displaystyle\frac{4x}{1+4 x^2}, &\phantom{0<}\,\,x>\tfrac12,
		\end{cases}
	\end{equation}
	which is discontinuous at $x=0$, explicit ODE solutions confirm that $V$ satisfies the nonincreasing condition~\cref{e:cond1-discontinuous}. This $V$ implies sharp bounds on $\maxphiinf$ for all sets $X_0$ of initial conditions, and in fact it is exactly the optimal $V^*$ defined by~\cref{e:value-function}.	
	\markendexample\end{example}


When trajectories are not known explicitly, the $V^*$ defined by~\cref{e:value-function} cannot be used to find explicit bounds, but it can still be useful. For instance, in \cref{s:direct-proof-strong-duality} we prove \cref{th:strong-duality} by showing that $V^*$ can be approximated with differentiable $V$. Moreover, $V^*$ has arisen in various contexts. One field in which $V^*$ arises is optimal control theory. Using ideas from dynamic programming for optimal stopping problems (see, e.g., section III.4.2 in~\cite{Bardi1997}) one can show that if $V^*$ is bounded and uniformly continuous on $\Omega$, then it is exactly the so-called value function for problem~\cref{e:maxphi} and is the unique viscosity solution to its corresponding Hamilton--Jacobi--Bellman complementarity system. This system consists of the auxiliary function constraints~{(\ref{e:V-conditions}a,b)} and the condition
\begin{equation}
	\mathcal{L}V(t,x)[\Phi(t,x) - V(t,x)] = 0   \quad\forall (t,x) \in \Omega.
	\label{e:hjb-complementarity}
\end{equation}
The auxiliary function framework studied in this work therefore can be seen as a relaxation of the Hamilton--Jacobi--Bellman system that results from dropping~\cref{e:hjb-complementarity}. A second connection between $V^*$ and existing literature occurs in the particular case of linear dynamics on a Hilbert space, as explained in the following example.

\begin{example}
	Let $X$ be a Hilbert space with inner product $\langle\cdot,\cdot \rangle$. Consider the autonomous \emph{linear} dynamical system $\dot{x} = A x$ with initial condition $x(0)=x_0$, where $A$ is a closed and densely defined linear operator, not necessarily bounded, that generates a strongly continuous semigroup $\{S_t\}_{t \geq 0}$. Trajectories satisfy $x(t) = S_t\,x_0$, so $S_t$ is the flow map. Suppose $S_t$ is compact for each $t>0$. In various linear systems of this type, one is interested in the maximum possible amplification of the norm $\|x\| = \sqrt{\langle x,x\rangle}$, which in {the} present framework means that $\Phi(x)=\|x\|$ with the initial set $X_0=\{x_0\in X:\,\|x_0\|=1\}$. In fluid mechanics, for instance, such problems have been studied to understand linear mechanisms by which perturbations are amplified (see, e.g.,~\cite{Trefethen1993}). With the above choices,~\cref{e:value-function} and~\cref{e:strong-duality-discontinuous} reduce to the well-known result
	\begin{equation}
		\Phi^*_\infty
		= \adjustlimits \sup_{\|x_0\|=1} \sup_{t\geq 0} \, \Phi(S_t\,x_0)
		= \adjustlimits \sup_{t\geq 0} \sup_{\|x_0\|=1}\, \sqrt{\langle S_t\,x_0, S_t\,x_0 \rangle}
		= \sup_{t \geq 0} \,  \sigma_{\rm max}(S_t),
	\end{equation}
	where $\sigma_{\rm max}(S_t)$ denotes the maximum singular value of $S_t$. We stress, however, that {the} general bounding framework {of \cref{ss:framework}} does not require an explicit flow map and applies also to nonlinear systems.
	\markendexample\end{example}

\section{Optimal trajectories}
\label{s:optimal-trajectories}

So far we have presented a framework for bounding the magnitudes of extreme events without finding the extremal trajectories themselves. The latter is much harder in general, partly due to the non-convexity of searching over initial conditions. However, auxiliary functions producing bounds on $\maxphi$ do give some information about optimal trajectories. Specifically, sublevel sets of any auxiliary function define regions of state space in which optimal and near-optimal trajectories must spend a certain fraction of time prior to the extreme event. A similar connection has been found between trajectories that maximize infinite-time averages and auxiliary functions that give bounds on these averages~\cite{Tobasco2018,Korda2018a}.
The following discussion applies to both global and local auxiliary functions with either finite or infinite time horizons. The simpler case of exactly optimal auxiliary functions is addressed in \cref{s:optimal-V}, followed by the general case in \cref{s:suboptimal-V}.

\subsection{Optimal auxiliary functions}
\label{s:optimal-V}

Suppose for now that the optimal bound~\cref{e:weak-duality-incomplete} is sharp and is attained by some $V^*$, in which case
\begin{equation}
	\label{e:optimal-af-definition}
	\sup_{x_0 \in X_0} V^*(t_0,x_0) = 
	\maxphi.
\end{equation}
Let $x_0^*\in X_0$ be an initial condition leading to an optimal trajectory, which attains the maximum value $\maxphi$ at some time $t^*$. To determine the value of $V^*$ on an optimal trajectory, note that the same reasoning leading to~\cref{e:weak-duality-incomplete} yields
\begin{align}
	\label{e:Vopt-inequalities}
	\maxphi
	&= \Phi[t,x(t^*;x_0^*)]  \\
	&\leq V^*(t_0,x_0^*) + \int_{t_0}^{t^*} \mathcal{L}V^*[\xi, x(\xi \given t_0,x_0^*)] \dxi \notag \\
	&\leq \sup_{x_0 \in X_0} V^*(t_0,x_0) + \int_{t_0}^{t^*} \mathcal{L}V^*[\xi, x(\xi \given t_0,x_0^*)] \dxi \notag \\
	&= \maxphi+ \int_{t_0}^{t^*} \mathcal{L}V^*[\xi, x(\xi \given t_0,x_0^*)] \dxi \notag \\
	&\leq 
	\maxphi \notag
\end{align}
The above inequalities must be equalities and $\mathcal{L}V^* \leq 0$, so $\mathcal{L}V^*\equiv0$ and $V^*\equiv\maxphi$ along an optimal trajectory up to time $t^*$. These constant values of $\mathcal{L}V^*$ and $V^*$ can be used to define sets in which optimal trajectories must lie:
\begin{align}
	\label{e:R0}
	\mathcal{R}_0 &:= \left\{(t,x)\in\Omega :\, \mathcal{L}V^*(t,x) = 0 \right\}, \\
	\label{e:S0}
	\mathcal{S}_0 &:= \left\{(t,x)\in\Omega:\, V^*(t,x) = \sup_{x_0 \in X_0} V^*(t_0,x_0) \right\},
\end{align}
where we have used \cref{e:optimal-af-definition} in defining $\mathcal{S}_0$. The intersection $\mathcal{S}_0 \cap \mathcal{R}_0$ contains the graph of each optimal trajectory until the last time that trajectory attains the maximum value $\maxphi$. In general, $\mathcal{S}_0 \cap \mathcal{R}_0$ may also contain points not on any optimal trajectory.

\subsection{General auxiliary functions}
\label{s:suboptimal-V}

Consider an auxiliary function $V$ and an initial condition $x_0$ that are a near-optimal pair, meaning that an upper bound on $\maxphi$ implied by $V$ and a lower bound implied by the trajectory starting from $x_0$ differ by no more than~$\delta$. That is, calling the upper bound $\lambda$,
\begin{equation}
	\label{e:delta-suboptimal-V}
	\lambda-\delta\le\sup_{t\in\mathcal{T}}\Phi[t,x(t;t_0,x_0)] \le \maxphi \le \sup_{x_0 \in X_0} V(t_0,x_0) \leq \lambda.
\end{equation}
The upper bound $\lambda$ might be larger than $\sup_{x \in X_0} V(t_0,x)$ if the latter cannot be computed exactly, and the lower bound $\lambda-\delta$ might be smaller than $\sup_{t\in\mathcal{T}}\Phi[t,x(t;t_0,x_0)]$ if the trajectory starting from $x_0$ is only partly known.

Let $t^*$ denote the latest time during the interval $\mathcal{T}$ when the trajectory starting at $x_0$ attains or exceeds the value $\lambda-\delta$. The constraints~(\ref{e:V-conditions}a,b) require $V$ to decay along trajectories and bound $\Phi$ pointwise, so
\begin{equation}
	\lambda-\delta
	\leq V[t^*,x(t^* \given t_0,x_0)]
	\leq V[t,x(t\given t_0,x_0)]
	\leq V(t_0,x_0)
	\leq \sup_{x \in X_0} V(t_0,x)
	\leq \lambda
\end{equation}
for all $t\in[t_0,t^*]$. The above inequalities imply that the trajectory starting at $x_0$ satisfies
\begin{equation}
	0 \leq \lambda - V[t, x(t\given t_0,x_0)] \leq \delta
\end{equation}
up to time $t^*$, so its graph must be contained in the set
\begin{equation}
	\label{e:Sdelta}
	\mathcal{S}_{\delta} := \left\{ (t,x) \in\Omega :\, 0 \leq \lambda - V(t, x) \leq \delta \right\},
\end{equation}
which extends to suboptimal $V$ the definition \cref{e:S0} of $\mathcal{S}_0$ for optimal $V^*$.

The definition \cref{e:R0} of $\mathcal{R}_0$ also can be extended to suboptimal $V$, but the resulting sets are guaranteed to contain optimal and near-optimal trajectories only for a certain amount of time. When $V$ satisfies~\cref{e:delta-suboptimal-V}, an argument similar to~\cref{e:Vopt-inequalities} shows that
\begin{equation}
	\maxphi \leq \maxphi + \delta + \int_0^{t^*}  \mathcal{L}V[\xi, x(\xi \given t_0,x_0)] \dxi,
\end{equation}
and therefore
\begin{equation}
	-\int_{t_0}^{t^*}  \mathcal{L}V[\xi, x(\xi \given t_0,x_0)] \dxi \leq \delta.
\end{equation}
Since $\mathcal{L}V\le0$, the above condition can be combined with Chebyshev's inequality (cf.\ \S VI.10 in~\cite{Knapp2005basic}) to estimate, for any $\varepsilon>0$, the total time during $[t_0,t^*]$ when $\mathcal{L}V\le-\varepsilon$. Letting $\Theta_\varepsilon$ denote this total time and letting $\mathbbm{1}_A$ denote the indicator function of a set $A$, we find
\begin{equation}
	\Theta_{\varepsilon}
	:=\int_{t_0}^{t^*} \mathbbm{1}_{ \{\xi:\,\mathcal{L}V[\xi, x(\xi \given t_0,x_0)] < -\varepsilon \} } \dxi
	\leq -\frac{1}{\varepsilon} \int_{t_0}^{t^*} \mathcal{L}V[\xi, x(\xi \given t_0,x_0)] \dxi
	\leq \frac{\delta}{\varepsilon}.
\end{equation}
In other words, a trajectory on which $\Phi\ge \lambda-\delta$ at some time $t^*$ cannot leave the set
\begin{equation}
	\label{e:Repsilon}
	\mathcal{R}_\varepsilon := \left\{ (t,x)\in\Omega :\,  -\varepsilon \leq \mathcal{L}V(t,x) \leq 0  \right\}
\end{equation}
for longer than $\delta/\varepsilon$ time units during the interval $[t_0,t^*]$. This statement is most useful when the upper bound $\maxphi\le\lambda$ implied by $V$ is close to sharp, so there exist trajectories where $\Phi$ attains values $\lambda-\delta$ with small $\delta$. Then one may take $\varepsilon$ small enough for $\mathcal{R}_\varepsilon$ to exclude much of state space, while also having it be meaningful that near-optimal trajectories cannot leave $\mathcal{R}_\varepsilon$ for longer than $\delta/\varepsilon$. The computational construction of $\mathcal S_{\delta}$ and $\mathcal{R}_\varepsilon$ for a polynomial ODE is illustrated by \cref{ex:sos-2d-example} in the next section.

\section{Computing bounds for ODEs using SOS optimization}
\label{s:sos-optimization}

The optimization of auxiliary functions and their corresponding bounds is prohibitively difficult in many cases, even by numerical methods. However, computations often are tractable when the system~\cref{e:system} is an ODE with polynomial righthand side $F:\R \times \R^n \to \R^n$, the observable $\Phi$ is polynomial, and the set of initial conditions $X_0$ is a basic semialgebraic set:
\begin{equation}
	\label{e:X0-semialg}
	X_0 := \{ x \in \R^n :\, f_1(x)\geq 0,\,\ldots, f_p(x) \geq 0, \,g_1(x)= 0,\,\ldots, g_q(x)=0 \}
\end{equation}
for given polynomials $f_1,\,\ldots,\,f_p$ and $g_1,\ldots,\,g_q$. The set $\Omega \subset \R \times \R^n$ in which the graphs of trajectories remain over the time interval $\mathcal{T}$ is assumed to be basic semialgebraic as well:
\begin{equation}
	\label{e:omega-semialg}
	\Omega := \{ (t,x) \in \R \times \R^n :\,h_1(t,x)\geq 0,\,\ldots, h_r(t,x) \geq 0, \,\ell_1(t,x)= 0,\,\ldots, \ell_s(t,x)=0 \}
\end{equation}
for given polynomials $h_1,\,\ldots,\,h_r$ and $\ell_1,\ldots,\,\ell_s$. To construct global auxiliary functions with state space $\R^n$, the set $\Omega$ can be specified by a single inequality: $h_1(t,x):=t-t_0\ge0$ or $h_1(t,x):=(t-t_0)(T-t)\ge0$ for infinite or finite time horizons, respectively. To construct local auxiliary functions, more inequalities or equalities must be added to define a smaller $\Omega$.

For any integer $d$, let $\R_{d}[t,x]$ and $\R_{d}[x]$ denote the vector spaces of real polynomials of degree $d$ or smaller in the variables $(t,x)$ and $x$, respectively. Restricting the optimization over differentiable auxiliary functions in~\cref{e:weak-duality} to polynomials in $\R_{d}[t,x]$ gives
\begin{equation}
	\label{e:weak-duality-polynomial}
	\maxphi
	\leq \inf_{\substack{V \in \R_{d}[t,x]\\\text{s.t. (\ref{e:V-conditions}a,b)}} }  \sup_{x_0 \in X_0} V(t_0, x_0).
\end{equation}
Recalling that the supremum over $X_0$ is the smallest upper bound $\lambda$ on that set, and substituting expression~\cref{e:LV-odes} for $\mathcal{L}V$ in the ODE case into~\cref{e:cond1}, we can express the righthand side of~\cref{e:weak-duality-polynomial} as a constrained minimization over $V$ and $\lambda$:
\begin{align}
	\label{e:sos-opt-partial}
	\maxphi
	\leq \inf_{\substack{V \in \R_{d}[t,x]\\\lambda \in \R}} \, \{ \lambda :\; 
	-\partial_t V(t,x) - F(t,x) \cdot \nabla_x V(t,x) &\geq 0  \text{ on } \Omega, \\[-1.35\fsize]
	V(t,x) - \Phi(t,x) &\geq 0  \text{ on } \Omega, \notag \\
	\lambda-V(t_0,x) &\geq 0 \text{ on } X_0 \}.\notag
\end{align}
Under the assumptions outlined above, the three constraints on $V$ and $\lambda$ are polynomial inequalities on basic semialgebraic sets. Checking such constraints is NP-hard in general~\cite{Murty1987}, so a common strategy is to replace them with stronger but more tractable constraints. Here we require that the polynomials in~\cref{e:sos-opt-partial} admit weighted sum-of-squares (WSOS) decompositions, which can be searched for computationally by solving~SDPs. These WSOS constraints imply that the inequalities in~\cref{e:sos-opt-partial} hold on $\Omega$ or $X_0$ but not necessarily outside these sets.

To define the relevant WSOS decompositions, let $\Sigma_{\mu}[t,x]$ and $\Sigma_{\mu}[x]$ be the cones of SOS polynomials of degrees up to $\mu$ in the variables $(t,x)$ and $x$, respectively. That is, a polynomial $\sigma\in\R_\mu[x]$ belongs to $\Sigma_\mu[x]$ if and only if there exist a finite family of polynomials $q_1,\,\ldots,\,q_k \in \R_{\lfloor \mu/2\rfloor}[x]$ such that $\sigma = \sum_{i=1}^k q_i^2$. For each integer $\mu$ that is no smaller than the highest polynomial degree appearing in the definition \cref{e:X0-semialg} of $X_0$, the set of degree-$\mu$ WSOS polynomials associated with $X_0$ is
\begin{align}
	\Lambda_\mu := \Big\{ 
	\sigma_0 + \sum_{i=1}^p f_i \sigma_i + \sum_{i=1}^q g_i \rho_i :\; \sigma_0 &\in \Sigma_\mu[x],  \\[-1.5ex]
	\sigma_i &\in \Sigma_{\mu-\deg(f_i)}[x], \;i=1,\,\ldots,\,p\notag \\
	\rho_i &\in \R_{\mu-\deg(g_i)}[x], \;i=1,\,\ldots,\,q
	\,\Big\}. \notag
\end{align}
In words, WSOS polynomials associated with $X_0$ can be written as a weighted sum of polynomials, where the weights are $\{1,f_1,\ldots,f_p,g_1,\ldots,g_q\}$ and the polynomials weighted by $\{1,f_1,\ldots,f_p\}$ are SOS. Every SOS polynomial is globally nonnegative, and it is WSOS with respect to any $X_0$ since all terms in the WSOS decomposition aside from $\sigma_0$ can be zero. On the other hand, WSOS polynomials need not be SOS.

Analogously to $\Lambda_\mu$, the set of degree-$\mu$ WSOS polynomials associated with $\Omega$ is
\begin{align}
	\Gamma_\mu := \Big\{ 
	\sigma_0 + \sum_{i=1}^r h_i \sigma_i + \sum_{i=1}^s \ell_i \rho_i :\; \sigma_0 &\in \Sigma_\mu[t,x],  \\[-1.5ex]
	\sigma_i &\in \Sigma_{\mu-\deg(h_i)}[t,x], \;i=1,\,\ldots,\,r\notag \\
	\rho_i &\in \R_{\mu-\deg(\ell_i)}[t,x], \;i=1,\,\ldots,\,s
	\,\Big\}.\notag
\end{align}
If a polynomial belongs to $\Gamma_\mu$ or $\Lambda_\mu$, then it is nonnegative on $\Omega$ or $X_0$, respectively. (The converse is false beyond a few special cases~\cite{Hilbert1888}.) We can strengthen the inequality constraints on $V$ in~\cref{e:sos-opt-partial} by requiring WSOS representations instead of nonnegativity. This gives
\begin{align}
	\label{e:sos-opt}
	\maxphi
	\leq \lambda^*_d 
	&:= \inf_{\substack{V \in \R_{d}[t,x]\\\lambda \in \R}} \, \{ \lambda :\; 
	-\partial_t V - F \cdot \nabla_{x}V \in \Gamma_{d-1+\deg(F)}, \\[-1.35\fsize]
	&\hspace{131pt}V - \Phi \in \Gamma_{d}, \notag \\
	&\hspace{108pt} \lambda-V(t_0,\cdot) \in \Lambda_d \}. \notag
\end{align}
For each integer $d$, the righthand side is a finite-dimensional optimization problem with WSOS constraints that are linear in the decision variables---the scalar $\lambda$ and the coefficients of the polynomial $V$. It is well known that such problems can be reformulated as SDPs (e.g., Section 2.4 in~\cite{Lasserre2015}). Such SDPs can be solved numerically in polynomial time, barring problems with numerical conditioning. Open-source software is available to assist both with the reformulation of WSOS optimizations as SDPs and with the solution of the latter.\footnote{Most modeling toolboxes for polynomial optimization, including the ones used in this work, do not natively support WSOS constraints. However, these can be implemented using standard SOS constraints. For instance, the WSOS constraint $P \in \Gamma_\mu$ can be implemented as the SOS constraint $P - \sum_{i=1}^p h_i \sigma_i - \sum_{i=1}^q \ell_i \rho_i \in \Sigma_\mu[t,x]$, along with the SOS constraints $\sigma_i \in \Sigma_{\mu - \deg(h_i)}[t,x]$ for $i=1,\ldots,p$. This formulation, known as the generalized S-procedure~\cite{Tan2006,Fantuzzi2016siads}, introduces more decision variables than the direct WSOS approach of~\cite[Section 2.4]{Lasserre2015}. The additional variables may lead to larger computations, but they can improve numerical conditioning by giving more freedom for the rescaling that is done within SDP solvers.} The SOS computations in \cref{ex:nonautonomous-example-sos,ex:sos-2d-example,ss:vdp}, and in \cref{app:iterative-procedure}, were set up in MATLAB using \yalmip~\cite{Lofberg2004,Lofberg2009} or a customized version of \spotless.\footnote{\href{https://github.com/aeroimperial-optimization/aeroimperial-spotless}{https://github.com/aeroimperial-optimization/aeroimperial-spotless}} The resulting SDPs were solved with the interior-point solver \mosek\ v.8~\cite{mosek} except in \cref{ss:vdp}, where the SDP was solved in multiple precision arithmetic with \sdpagmp\ v.7.1.3~\cite{sdpagmp}.

The bounds $\lambda^*_d$ found by solving~\cref{e:sos-opt} numerically form a nonincreasing sequence as the degree $d$ of $V$ is raised. These bounds appear to become sharp in various cases, including \cref{ex:nonautonomous-example-sos} above and \cref{ex:sos-2d-example} below. We cannot say whether such convergence occurs in all cases, even when auxiliary functions arbitrarily close to optimality are known to exist. This is due to our restriction to polynomial $V$ and use of WSOS constraints, which are sufficient but not necessary for nonnegativity. However, if the sets $X_0$ and $\Omega$ are both compact and there exists a differentiable $V$ attaining equality in~\cref{e:weak-duality}, then the following theorem guarantees that bounds from SOS computations become sharp as the polynomial degree is raised. The proof is a standard argument in SOS optimization and relies on a result known as Putinar's Positivstellensatz~\cite[Lemma 4.1]{Putinar1993}, which guarantees the existence of WSOS representations for strictly positive polynomials; details can be found in Section 2.4 of~\cite{Lasserre2015}.
\begin{theorem}
	\label{th:sos-convergence}
	Let $\Omega$ and $X_0$ be compact semialgebraic sets. Assume the definitions of $\Omega$ and $X_0$ include inequalities $C_1-t^2 - \|x\|_2^2\ge0$ and $C_2-\|x\|_2^2\ge0$ for some $C_1$ and $C_2$, respectively, which can always be made true by adding inequalities that do not change the specified sets. Let $\lambda_d^*$ be the bound from the optimization~\cref{e:sos-opt}. If differentiable auxiliary functions give arbitrarily sharp bounds~\cref{e:strong-duality} on $\maxphiT$, then $\lambda_d^* \to \maxphiT$ as $d \to \infty$.
\end{theorem}
\begin{proof}
	Assume that the semialgebraic definitions of $\Omega$ and $X_0$ include inequalities of the form $C_1-t^2 - \|x\|_2^2\ge0$ and $C_2- \|x\|_2^2\ge0$, respectively. If not, these inequalities can be added with $C_1$ and $C_2$ large enough to not change which points lie in $\Omega$ and $X_0$ since both sets are compact. Then, $C_1-t^2 - \|x\|_2^2 \in \Gamma_\mu$ and $C_2-\|x\|_2^2 \in \Lambda_\mu$ for all integers $\mu$.\footnote{\Cref{th:sos-convergence} holds also when the semialgebraic definitions of $\Omega$ and $X_0$ satisfy Assumption 2.14 in~\cite[Section 2.4]{Lasserre2015}, which is a slightly weaker but more technical condition implying the inclusions $C_1-t^2 - \|x\|_2^2 \in \Gamma_\mu$ and $C_2-\|x\|_2^2 \in \Lambda_\mu$ for all sufficiently large integers $\mu$.}
	
	To prove that $\lambda_d^* \to \maxphiT$ as $d \to \infty$, we establish the equivalent claim that, for each $\varepsilon>0$, there exists an integer $d$ such that $\lambda_d^* \leq \maxphiT + \varepsilon$. Choose $\gamma>0$ such that
	\begin{equation}
		\label{e:sos-duality-gamma}
		\gamma < \frac{2T \varepsilon}{5T-t_0}.
	\end{equation}
	By assumption there exists an auxiliary function $W\in C^1(\Omega)$, not generally a polynomial, such that
	\begin{equation}
		\label{e:suboptimal-condition-sf-th2}
		W(t_0,x_0) \leq \maxphiT + \gamma \quad\text{on } X_0.
	\end{equation}
	Since $\Omega$ is compact, polynomials are dense in $C^1(\Omega)$ (cf.\ Theorem 1.1.2 in~\cite{Llavona1986}). That is, for each $\delta>0$ there exists a polynomial $P$ such that $\|W-P\|_{C^1(\Omega)} \leq \delta$, where $\|\cdot\|_{C^k(\Omega)}$ denotes the usual norm on $C^k(\Omega)$---the sum of the $L^\infty$ norms of all derivatives up to order $k$. Fix such a $P$ with
	\begin{equation}
		\label{e:sos-duality-delta}
		\delta < \frac{\gamma}{\max\left\{ 2,2T,2T\Vert F_1\Vert_{C^0(\Omega)},\ldots,2T\Vert F_n\Vert_{C^0(\Omega)} \right\}}.
	\end{equation}
	By definition $\Omega$ contains the initial set $\{t_0\}\times X_0$, so $\abs{W(t_0,\cdot)-P(t_0,\cdot)} < \delta$ uniformly on $X_0$. We define the polynomial auxiliary function
	\begin{equation}
		V(t,x) = P(t,x) + \gamma\left( 1- \frac{t}{2T}\right).
	\end{equation}
	With $\delta$ as in~\cref{e:sos-duality-delta}, $\gamma$ as in~\cref{e:sos-duality-gamma}, and $W$ satisfying~\cref{e:suboptimal-condition-sf-th2}, elementary estimates show that
	\begin{subequations}
		\label{e:strict-inequalities-V}
		\begin{align}
			-\partial_t V - F \cdot \nabla_{x}V &> 0 \quad\text{on } \Omega,\\
			V - \Phi &> 0 \quad\text{on } \Omega,\\
			\maxphiT + \varepsilon -V(t_0,\cdot) &> 0 \quad\text{on } X_0.
			\label{e:strict-inequalities-V-c}
		\end{align}
	\end{subequations}
	
	The inequalities~(\ref{e:strict-inequalities-V}a--c) are strict. Since $C_1-t^2 - \|x\|_2^2 \in \Gamma_\mu$ and $C_2-\|x\|_2^2 \in \Lambda_\mu$ for all integers $\mu$ by assumption, a straightforward corollary of Putinar's Positivstellensatz~\cite[Lemma 4.1]{Putinar1993} guarantees that inequalities~(\ref{e:strict-inequalities-V}a--c) can be proved with WSOS certificates. Precisely, there exists an integer $\mu'$ such that the polynomials in~(\ref{e:strict-inequalities-V}a,b) belong to $\Gamma_{\mu'}$, and the polynomial in~\cref{e:strict-inequalities-V-c} belongs to $\Lambda_{\mu'}$. We now set $d = \max\{\deg(V),\mu'\}$ and observe that $V$ is feasible for the righthand problem in~\cref{e:sos-opt} with $\lambda=\maxphiT+\varepsilon$ because $\Gamma_{\mu'} \subseteq \Gamma_d$, $\Lambda_{\mu'} \subseteq \Lambda_d$, and $V \in \R_d[t,x]$. This proves the claim that $\lambda_d^* \leq \maxphiT+\varepsilon$.
\end{proof}

The computational cost of solving WSOS optimization problems grows quickly as $d$ is raised. For instance, suppose the polynomials $f_1,\,\ldots,\,f_p$ and $h_1,\,\ldots,\,h_r$ all have the same degree $\omega$, and let $d_F:=d-1+\deg(F)$. Then, the time for standard primal-dual interior-point methods scales as $\mathcal{O}( L_1^{6.5} + (p+r)^{1.5} L_2^{6.5})$, where $L_1 = \binom{n+\lfloor d_F/2 \rfloor}{n}$ and $L_2 = \binom{n+\lfloor (d-\omega)/2 \rfloor}{n}$; see~\cite{Papp2019} and references therein for further details. \Cref{app:iterative-procedure} describes a way to improve bounds iteratively without raising $d$, but the improvement is small in the example tested. Poor computational scaling with increasing $d$ can be partly mitigated if symmetries of optimal $V$ can be anticipated and enforced in advance, leading to smaller SDPs. When the differential equations, the observable $\Phi$, and the sets $\Omega$ and $X_0$ all are invariant under a symmetry transformation, then the optimal bound is unchanged if the symmetry is imposed also on $V$ and the weights $\sigma_i$ and $\rho_i$. The next proposition formalizes these observations; its proof is a straightforward adaptation of a similar result in Appendix A of~\cite{Goluskin2019}, so we do not report it.

\begin{proposition}
	\label{th:symmetry-reduction}
	Let $A : \R^{n \times n}$ be an invertible matrix such that $A^k$ is the identity for some integer $k$. Assume that $F(t,A x) = A F(t,x)$, $\Phi$ is $A$-invariant in the sense that $\Phi(t,A x)= \Phi(t,x)$, and all polynomials defining $\Omega$ and $X_0$ are $A$-invariant also.
	If $V \in \V(\Omega)$ gives a bound $\maxphi \leq \lambda$, then there exits $\widehat{V}\in \V(\Omega)$ that is $A$-invariant and proves the same bound. Moreover, if the pair $(V,\lambda)$ satisfies the WSOS constraints in~\cref{e:sos-opt}, then so does the pair $(\widehat{V},\lambda)$ and there exist WSOS decompositions with $A$-invariant weights $\sigma_i$, $\rho_i$.
\end{proposition}

We conclude this section with three computational examples. The first two demonstrate that SOS optimization can give extremely good bounds on both $\maxphiT$ and $\maxphiinf$ in practice, even when the assumptions of \cref{th:strong-duality,th:sos-convergence} do not hold. The first example also illustrates the approximation of optimal trajectories described in \cref{s:optimal-trajectories}. The third example, on the other hand, reveals a potential pitfall of SOS optimization applied to bounding $\maxphiinf$ for systems with periodic orbits: infeasible problems may appear to be solved successfully due to unavoidably finite tolerances in SDP solvers.

\begin{example}
	\label{ex:sos-2d-example}
	\belowpdfbookmark{Example~\ref{ex:sos-2d-example}}{bookmark:sos-2d-example}
	
	Consider the nonlinear autonomous ODE system
	\begin{equation}
		\label{e:ex-nonnormal-2d}
		\begin{bmatrix}
			\dot{x}_1 \\ \dot{x}_2
		\end{bmatrix}
		= \begin{bmatrix}
			0.2 x_1 + x_2 - x_2(x_1^2 + x_2^2)\\
			-0.4 x_2 + x_1(x_1^2 + x_2^2)
		\end{bmatrix},
	\end{equation}
	which is symmetric under $x \mapsto -x$. As shown in \cref{f:sos-2d-example-phase-portrait}(a), the system has a saddle point at the origin and a symmetry-related pair of attracting equilibria. Let $X_0 = \{x: \|x\|_2^2=0.25\}$. Aside from two points on the stable manifold of the origin, all points in $X_0$ produce trajectories that eventually spiral outwards towards the attractors, as shown in \cref{f:sos-2d-example-phase-portrait}(b). 
	
	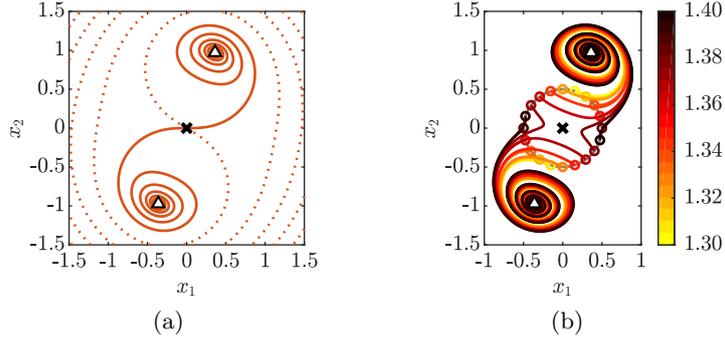
\begin{figure}
		\centering
		\includegraphics[scale=1]{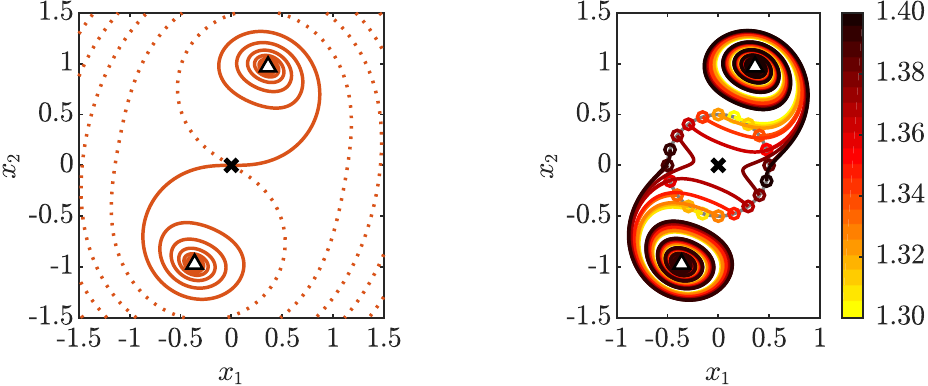}\\[-1ex]
		\begin{tikzpicture}[overlay]
		\node[fill=white] at (-2.6,-0.025) {\footnotesize(a)};
		\node[fill=white] at (2.65,-0.025) {\footnotesize(b)};
		\end{tikzpicture}
		\caption{(a) Phase portrait of the ODE~\cref{e:ex-nonnormal-2d} showing the attracting equilibria (\mytriangle{black}), the saddle (\mycross{black}), and the saddle's unstable ({\color{matlabred}\solidrule}) and stable ({\color{matlabred}\dottedrule}) manifolds. (b) Sample trajectories starting from the circle $\|x\|_2^2=0.25$. Small circles mark the initial conditions. Colors indicate the maximum value of $\Phi=\|x\|_2^2$ along each trajectory.}	
		\label{f:sos-2d-example-phase-portrait}
	\end{figure}
	\begin{figure}[t]
		\centering
		\includegraphics[scale=1]{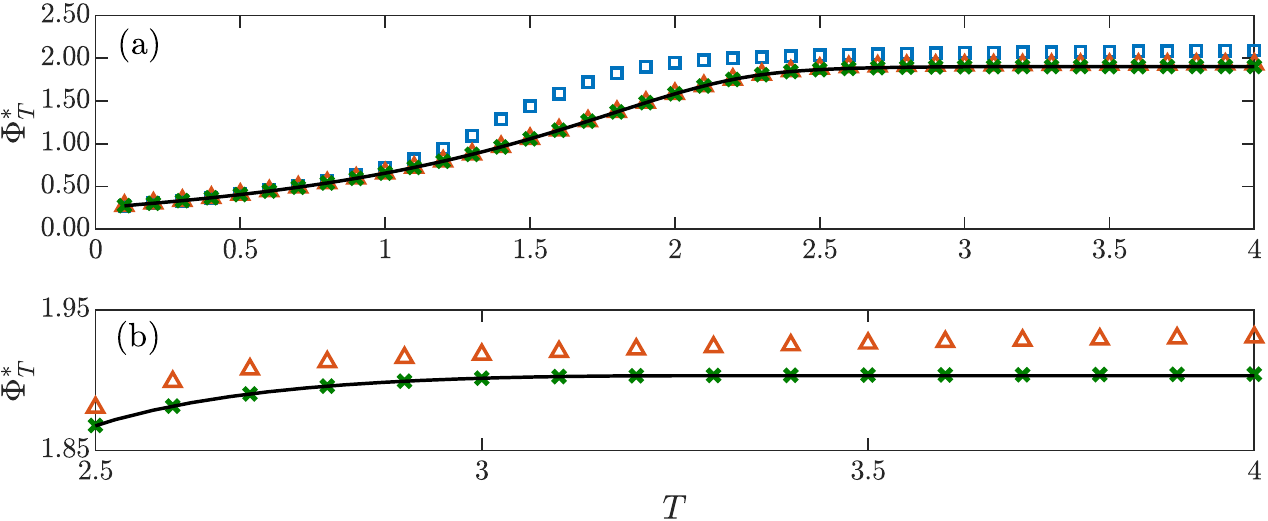}
		\caption{
			(a) Upper bounds on $\maxphiT$ in \cref{ex:sos-2d-example} for various time horizons $T$, computed using auxiliary functions $V(t,x)$ with polynomial degrees 4~(\mysquare{matlabblue}), 6~(\mytriangle{matlabred}), and 8~(\mycross{matlabgreen}). Lower bounds on $\maxphiT$ found by maximizing $\Phi[x(T\given 0, x_0)]$ over $x_0$ using adjoint optimization are also plotted~(\solidrule).
			(b) Detailed view of part of panel (a).}
		\label{f:sos-2d-example-finite-T-bounds}
	\end{figure}
	
	Using SOS optimization, we have computed upper bounds on the value of $\Phi(x)=\|x\|_2^2$ among all trajectories starting from $X_0$, for both finite and infinite time horizons. For simplicity we considered only global auxiliary functions, meaning we used $\Omega = [0,T] \times \R^2$ and $\Omega = [0,\infty) \times \R^2$ to solve~\cref{e:sos-opt} in the finite- and infinite-time cases, respectively. Since both choices of $\Omega$ and the set of initial conditions $X_0 = \{x: \|x\|_2^2=0.25\}$ share the same symmetry as~\cref{e:ex-nonnormal-2d}, we applied \cref{th:symmetry-reduction} to reduce the cost of solving~\cref{e:sos-opt}. Our implementation used \yalmip\ to reformulate~\cref{e:sos-opt} into an SDP, which was solved with \mosek.
	
	\Cref{f:sos-2d-example-finite-T-bounds} shows upper bounds on $\maxphiT$ that we computed for a range of time horizons $T$ by solving~\cref{e:sos-opt} with time-dependent polynomial $V$ of degrees $d=4$, 6, and 8. Also plotted in the figure are lower bounds on $\maxphiT$, found by searching among initial conditions using adjoint optimization. The close agreement with our upper bounds shows that the degree-8 bounds are very close to sharp, and that adjoint optimization likely has found the globally optimal initial conditions. We find that $\maxphiT = \maxphiinf \approx1.90318$ for all $T\ge3.2604$, indicating that $\Phi$ attains its maximum over all time when $T\approx 3.2604$.

	\begin{table}[t]
		\caption{Upper bounds on $\maxphiT$ and $\maxphiinf$ for \cref{ex:sos-2d-example}, computed by solving~\cref{e:sos-opt}. The bounds for $\maxphiT$ and $\maxphiinf$ were computed using time-dependent and time-independent $V$, respectively. Lower bounds are implied by the maximum of $\Phi$ on particular trajectories, whose initial conditions were found by adjoint optimization.}
		\label{t:sos-2d-example-bounds}
		\centering
		\small
		\begin{tabular}{c c c c c}
			\toprule
			& $\deg(V)$ & $T=2$ & $T=3$ & $T=\infty$ \\[2pt]
			\hline
			Upper bounds&4  	& 1.948016 	& 2.062952 	& 2.194343 \\
			&6   	& 1.584910 	&  1.918262 	& 1.942396 \\
			&8   	& 1.584055 	& 1.901411 	& 1.931330 \\
			&10  	& "			& 1.901409 	& 1.916228 \\
			&12  	& "           		& " 		 	& 1.903525 \\
			&14  	& "           		& " 	     		& 1.903448 \\
			&16  	& "    	       	& " 		 	& 1.903185 \\
			&18  	& " 	         	& " 		     	& 1.903181 \\	
			\hline
			Lower bounds && 1.584055 & 1.901409 & 1.903178 \\
			\bottomrule
		\end{tabular}
	\end{table}
	
	\Cref{t:sos-2d-example-bounds} reports upper bounds on $\maxphiT$ computed with time-dependent $V$ up to degree 18 for $T=2$ and $T=3$, as well as upper bounds on $\maxphiinf$. The infinite-time implementation was restricted to time-independent polynomial $V(x)$ because polynomial dependence on $t$ gave no improvement in preliminary computations. This restriction lowers the computational cost because the first two WSOS constraints in~\cref{e:sos-opt} are independent of time and reduce to standard SOS constraints on $\mathbb{R}^2$. The resulting bounds are excellent for each $T$ reported in \cref{t:sos-2d-example-bounds}. As the degree of $V$ is raised, the upper bounds on $\maxphi$ apparently converge to the lower bounds produced by adjoint optimization. Note that this convergence is not guaranteed by \cref{th:strong-duality,th:sos-convergence} because the domain $\Omega$ is not compact. 
	
	Finally, we illustrate how auxiliary functions can be used to localize optimal trajectories using the methods described in \cref{s:optimal-trajectories}. For a near-optimal $V$ we take the time-independent degree-$14$ auxiliary function that gives the upper bound $\lambda=1.903448$ reported in \cref{t:sos-2d-example-bounds}. Any trajectory that attains or exceeds a value $\lambda-\delta$ at some time $t^*$ must spend the interval $[t_0,t^*]$ inside the set $\mathcal{S}_\delta$ defined by \cref{e:Sdelta}. In the present example, the lower bound $1.903178\le\maxphi$ guarantees the existence of such trajectories for all $\delta\ge0.00027$. In general a good lower bound on $\maxphi$ may be lacking, in which case the sets $\mathcal{S}_\delta$ tell us where near-optimal trajectories must lie \emph{if} they exist. With this general situation in mind, \cref{f:sos-2d-example-bounding-sets}(a,b) show $\mathcal{S}_\delta$ for $\delta=0.01$ and $0.002$, along with the exactly optimal trajectories. The $\mathcal{S}_\delta$ sets localize the optimal trajectories increasingly well as $\delta$ is lowered, although they contain other parts of state space also. \Cref{f:sos-2d-example-bounding-sets}(c) shows the sets $\mathcal{R}_\varepsilon$, defined by \cref{e:Repsilon}, for $\varepsilon=0.008$ and 0.004. Each trajectory coming within $\delta=0.002$ of the upper bound, for example, cannot leave these $\mathcal{R}_\varepsilon$ for longer than $\delta/\varepsilon=0.25$ and $0.5$ time units, respectively, prior to any time at which $\Phi\ge\lambda-\delta$. The same is true of the intersections of these sets with $\mathcal{S}_\delta$, which are shown in \cref{f:sos-2d-example-bounding-sets}(d).
	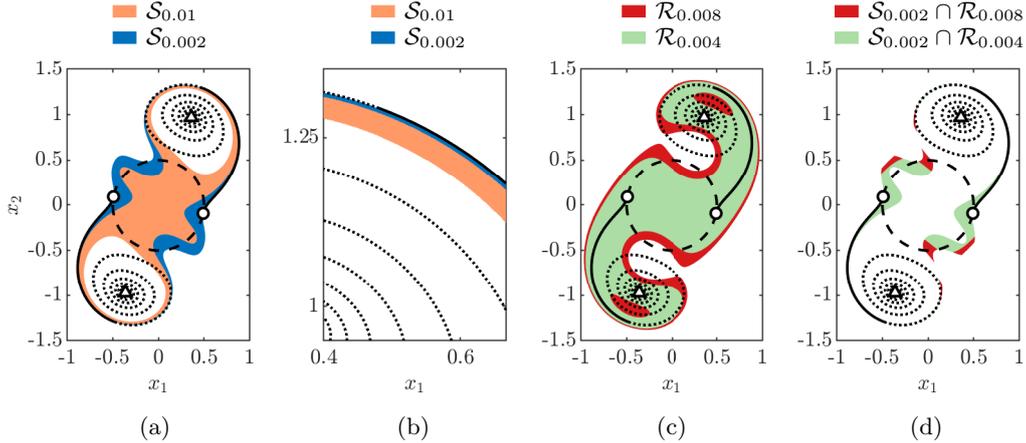
\begin{figure}
		\centering
		\vspace*{4ex}
		\includegraphics[scale=1]{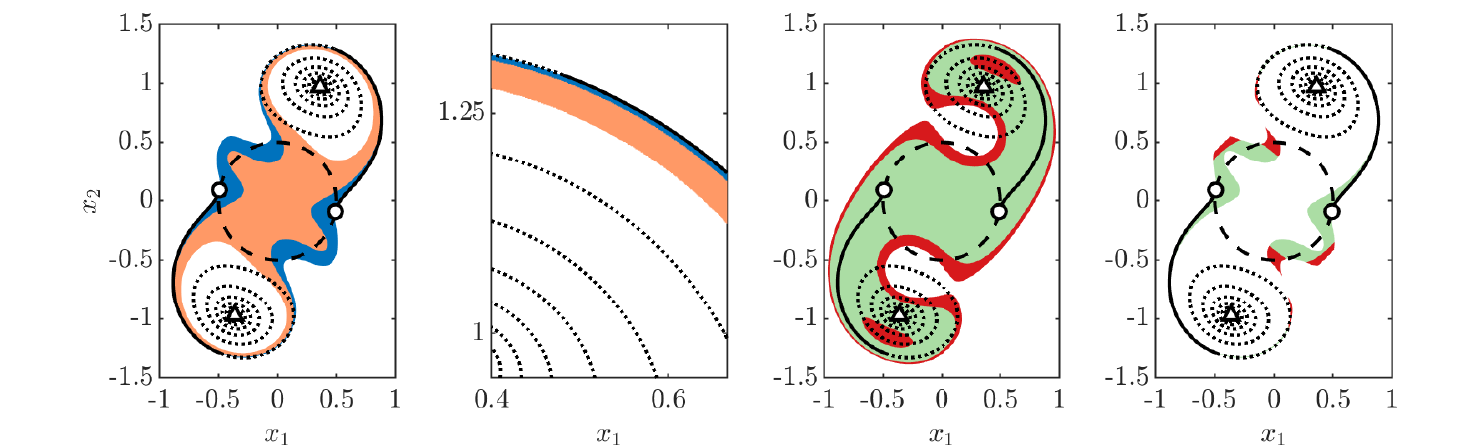}\\
		\begin{tikzpicture}[overlay]
		\node[fill=white] at (-4.75,-0.0) {\footnotesize(a)};
		\node[fill=white] at (-1.35,-0.0) {\footnotesize(b)};
		\node[fill=white] at (2.05,-0.0) {\footnotesize(c)};
		\node[fill=white] at (5.4,-0.0) {\footnotesize(d)};
		\draw[fill=matlabblue,draw=matlabblue] (-5.3,5.1) rectangle (-5,5.25);
		\draw[fill=matlaborange,draw=matlaborange] (-5.3,5.45) rectangle (-5,5.6);
		\node[anchor=west] at (-5,5.175) {\footnotesize$\mathcal{S}_{0.002}$};
		\node[anchor=west] at (-5,5.525) {\footnotesize$\mathcal{S}_{0.01}$};
		\draw[fill=matlabblue,draw=matlabblue] (-1.9,5.1) rectangle (-1.6,5.25);
		\draw[fill=matlaborange,draw=matlaborange] (-1.9,5.45) rectangle (-1.6,5.6);
		\node[anchor=west] at (-1.6,5.175) {\footnotesize$\mathcal{S}_{0.002}$};
		\node[anchor=west] at (-1.6,5.525) {\footnotesize$\mathcal{S}_{0.01}$};
		\draw[fill=matlabsafegreen,draw=matlabsafegreen] (1.4,5.1) rectangle (1.7,5.25);
		\draw[fill=matlabsafered,draw=matlabsafered] (1.4,5.45) rectangle (1.7,5.6);
		\node[anchor=west] at (1.7,5.175) {\footnotesize$\mathcal{R}_{0.004}$};
		\node[anchor=west] at (1.7,5.525) {\footnotesize$\mathcal{R}_{0.008}$};
		\draw[fill=matlabsafegreen,draw=matlabsafegreen] (4.2,5.1) rectangle (4.5,5.25);
		\draw[fill=matlabsafered,draw=matlabsafered] (4.2,5.45) rectangle (4.5,5.6);
		\node[anchor=west] at (4.5,5.175) {\footnotesize$\mathcal{S}_{0.002} \cap \mathcal{R}_{0.004}$};
		\node[anchor=west] at (4.5,5.525) {\footnotesize$\mathcal{S}_{0.002} \cap \mathcal{R}_{0.008}$};
		\end{tikzpicture}
		\\
		\caption{Sets approximating the trajectories that attain $\maxphiinf$ for \cref{ex:sos-2d-example}: 
			(a)~$\mathcal{S}_{0.01}$ and  $\mathcal{S}_{0.002}$. 
			(b)~Detail view of part of panel (a). 
			(c)~$\mathcal{R}_{0.008}$ and  $\mathcal{R}_{0.004}$. 
			(d)~$\mathcal{S}_{0.002}\cap \mathcal{R}_{0.008}$ and  $\mathcal{S}_{0.002} \cap \mathcal{R}_{0.004}$. 
			All sets were computed using the same degree-$14$ polynomial $V(x)$ that yields the nearly sharp bounds in \cref{t:sos-2d-example-bounds}. Also plotted are the attracting equilibria (\mytriangle{black}), the set of initial conditions $X_0$ ({\color{black}\dashedrule}), the optimal initial conditions (\mycirc{black}), and the optimal trajectories before (\solidrule) and after (\dottedrule) the point at which $\maxphiinf$ is attained.}
		\label{f:sos-2d-example-bounding-sets}
	\end{figure}
	\markendexample\end{example}

\begin{example}
	\label{ex:burgers-sos-example}
	\belowpdfbookmark{Example~\ref{ex:burgers-sos-example}}{bookmark:burgers-sos-example}
	
	Here we consider a $16$-dimensional ODE model obtained by projecting the Burgers equation~\cref{e:fractional-burgers} with ordinary diffusion ($\alpha=1$) onto modes $u_n(x) = \sqrt{2}\sin(2 n \pi x)$, $n=1,\,\ldots,\,16$. In other words, we substitute the expansion $u(x,t) = \sum_{m=1}^{16} a_m(t) u_m(x)$ into~\cref{e:fractional-burgers} with $\alpha=1$ and integrate the result against each $u_n(x)$ to derive $16$ nonlinear coupled ODEs for the amplitudes $a_1(t),\,\ldots,\,a_{16}(t)$. This gives
	\begin{equation}
		\label{e:burgers-truncated-ode}
		\dot{a}_n = -\left(2 \pi n\right)^2 a_n +  \sqrt{2} \pi n \left[ \sum_{m=1}^{16-n} a_m a_{m+n} - \frac12 \sum_{m=1}^{n-1}a_m a_{n-m} \right], \qquad n=1,\,\ldots,\,16.
	\end{equation}
	
	Let $a=(a_1,\,\ldots,\,a_{16})$ denote the state vector. Similarly to what is done for the PDE in \cref{ex:fractional-burgers}, we bound the projected enstrophy $\Phi(a) := 2 \pi^2 \sum_{n=1}^{16} n^2 a_n^2$ along trajectories with initial conditions in the set $X_0 = \{a \in \mathbb{R}^{16}\,:\, \Phi(a)=\Phi_0\}$, and we consider various values $\Phi_0$ of the initial enstrophy. We construct time-independent degree-$d$ polynomial $V$ of the form
	\begin{equation}
		\label{e:burgers-ode-V}
		V(a) = c \|a\|_2^d + P_{d-1}(a),
	\end{equation}
	where $d$ is even, $c$ is a tunable constant, and $P_{d-1}(a)$ is a tunable polynomial of degree $d-1$. Since the nonlinear terms in~\cref{e:burgers-truncated-ode} conserve the leading $\|a\|_2^d$ term, $\mathcal{L}V$ has the same even leading degree as $V$, which is necessary for {(\ref{e:V-conditions}a,b)} to hold over the global spacetime set $\Omega = [0,\infty)\times \R^{16}$. We also construct local $V$ of the form~\cref{e:burgers-ode-V} by imposing {(\ref{e:V-conditions}a,b)} only on the smaller spacetime set $\Omega = [0,\infty) \times X$ with
	\begin{equation}
		X := \left\{ a \in \R^{16}: \|a\|_2^2 \leq \frac{\Phi_0}{2\pi^2} \right\}.
	\end{equation}
	All trajectories starting from $X_0$ remain in $X$ because~\cref{e:burgers-truncated-ode} implies $\ddt \|a\|_2^2 = -4 \Phi(a) \leq 0$, so $\|a\|_2^2$ is bounded by its initial value, and $\|a\|_2^2 \leq \frac{1}{2\pi^2} \Phi(a)$ pointwise.
	
	\begin{figure}[t]
		\centering
		\includegraphics[scale=0.9]{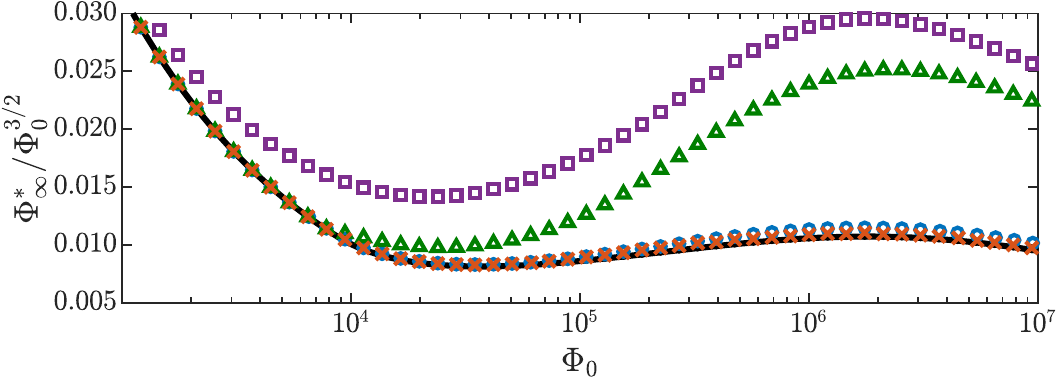}
		\caption{Bounds on $\maxphiinf$ for~\cref{e:burgers-truncated-ode} computed with both global and local polynomial auxiliary functions $V$ of the form~\cref{e:burgers-ode-V} for $d=4$ (\mysquare{matlabpurple}~global, \mytriangle{matlabgreen}~local) and $d=6$ (\mycirc{matlabblue}~global, \mycross{matlabred}~local). Also plotted are lower bounds on $\maxphiinf$ obtained with adjoint optimization (\solidrule). All results are normalized by $\Phi_0^{3/2}$, the expected scaling at large~$\Phi_0$~\cite{Ayala2011}.}
		\label{f:burgers-bounds}
	\end{figure}
	
	\Cref{f:burgers-bounds} shows upper bounds on $\maxphiinf$ computed for $\Phi_0$ values spanning four orders of magnitude using both global and local $V$ of degrees 4 and 6. Also shown are lower bounds obtained using adjoint optimization. (Note that the 16-mode truncation~\cref{e:burgers-ode-V} accurately resolves Burgers equation only in cases with $\Phi_0\lesssim2\cdot10^5$.) We used \spotless\ and \mosek\ to solve~\cref{e:sos-opt} and applied \cref{th:symmetry-reduction} to exploit symmetry under the transformation $a_n \mapsto (-1)^{n} a_n$. At each $\Phi_0$ value, constructing quartic $V$ required approximately 60 seconds on 4 cores with 16GB of memory. Local quartic $V$ produce better bounds than global ones, the results obtained with the former being within 1\% of the lower bounds from adjoint optimization for $\Phi_0 \lesssim 8000$. The results improve significantly with sextic $V$: for all tested $\Phi_0$, the upper bounds produced by global and local sextic $V$ are within 9\% and 5\% of the adjoint optimization results, respectively. Constructing sextic $V$ at a single $\Phi_0$ value required 16 hours on a 12-core workstation with 48GB of memory, which is significantly more expensive than adjoint optimization. However, we stress that auxiliary functions yield \textit{upper} bounds on $\maxphiinf$, while adjoint optimization gives only \textit{lower} bounds on $\maxphiinf$, so the two approaches give different and complementary results.
	\markendexample\end{example}

It is evident that SOS optimization can produce excellent bounds on extreme events given enough computational resources, but care must be taken to assess whether numerical results can be trusted. As observed already in the context of SOS optimization~\cite{Waki2012}, numerical SDP solvers can return solutions that appear to be correct but are provably not so. The next example shows that this issue can arise when bounding $\maxphiinf$ in systems with periodic orbits.

\begin{example}
	\label{ss:vdp}
	\belowpdfbookmark{Example~\ref{ss:vdp}}{bookmark:vdp} 
	
	Consider a scaled version of the van der Pol oscillator~\cite{VanderPol1926},
	\begin{equation}
		\label{e:vdp}
		\begin{bmatrix}
			\dot{x}_1\\ \dot{x}_2
		\end{bmatrix}
		= \begin{bmatrix}
			x_2\\
			(1-9x_1^2)x_2-x_1
		\end{bmatrix},
	\end{equation}
	which has a limit cycle attracting all trajectories except the unstable equilibrium at the origin (see \cref{f:vdp}). Let $\Phi = \|x\|_2^2$ be the observable of interest. We seek bounds on $\maxphiinf$ along trajectories starting from the circle $\|x\|_2^2 = 0.04$. All such trajectories approach the limit cycle from the inside, so $\maxphiinf$ coincides with the pointwise maximum of $\Phi$ on the limit cycle. Maximizing $\Phi$ numerically along the limit cycle yields $\maxphiinf \approx 0.889856$.
	
	\begin{figure}
		\centering
		\includegraphics[scale=1]{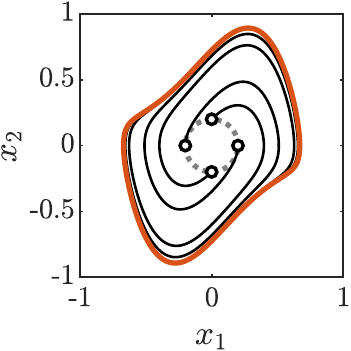}
		\caption{Limit cycle ({\color{matlabred}\solidrule}) for the scaled van der Pol oscillator~\cref{e:vdp}. Also plotted are trajectories (\solidrule) with initial conditions (\mycirc{black}) on the circle $\|x\|_2^2=0.04$ ({\color{matlabgray}\dashedrule}).}
		\label{f:vdp}
	\end{figure}
	\begin{table}
		\caption{Parameters for \sdpagmp\ used in \cref{ss:vdp} to produce an invalid degree-22 auxiliary function for the scaled van der Pol oscillator. A description of each parameter can be found in~\cite{sdpagmp}.}
		\label{t:sdpa-gmp-parameters}
		\centering
		\small
		\begin{tabular}{rl c rl c rl c rl}
			\toprule
			epsilonStar & $10^{-25}$ &&  betaStar & 0.1 && 	lowerBound & -$10^{25}$ && maxIteration & 200\\
			epsilonDash & $10^{-25}$ &&  betaBar & 0.3 && upperBound & \phantom{-}$10^{25}$ &&   precision & 200\\
			lambdaStar & $10^4$ && gammaStar & 0.7 && omegaStar & \phantom{-}2\\
			\bottomrule
		\end{tabular}
	\end{table}
	
	We implemented~\cref{e:sos-opt} with \yalmip\ using a time-independent polynomial auxiliary function $V(x)$ of degree 22. To confirm that difficulties were not easily avoided by increasing precision, we solved the resulting SDP in multiple precision arithmetic using the solver \sdpagmp\ v.7.1.3. The solver parameters we used are listed in \cref{t:sdpa-gmp-parameters} in order to ensure that our results are reproducible; see~\cite{sdpagmp} for the meaning of each parameter. The solver terminated successfully after 95 iterations, reporting no error and returning the upper bound $\maxphiinf \leq 0.956911$. Although this bound is true, it reflects an invalid SOS solution because no time-independent polynomial $V$ of any degree can satisfy~\cref{e:cond1}. To see this, suppose that~\cref{e:cond1} holds, so $V$ cannot increase along trajectories of~\cref{e:vdp}. In particular, if $x(t)$ lies on the limit cycle and $\tau$ is the period, then for all $\alpha \in (0,1)$,
	\begin{equation}
		V[x(t)] \geq V[x(t+\alpha \tau)] \geq V[x(t+\tau)] = V[x(t)].
	\end{equation}
	Thus, time-independent $V$ giving finite bounds on $\maxphiinf$ must be constant on the limit cycle. This is impossible if $V$ is polynomial because the limit cycle is not an algebraic curve~\cite{Odani1995}.
	
	There are two possible reasons why the SDP solver does not detect that the problem is infeasible despite the use of multiple precision. The first is that inevitable roundoff errors mean that our bound does not apply to~\cref{e:vdp}, but to a slightly perturbed system whose limit cycle \textit{is} an algebraic curve. The second possibility, which seems more likely, is that although no time-independent polynomial $V$ is feasible, there exists a feasible nonpolynomial $V$ that can be approximated accurately near the limit cycle by a degree-22 polynomial. In particular, the approximation error is smaller than the termination tolerances used by the solver, which therefore returns a solution that is not feasible but very nearly so. This interpretation is supported by the fact that \sdpagmp\ issues a warning of infeasibility when its tolerances are tightened by lowering the values of parameters epsilonDash and epsilonStar to $10^{-30}$.
	\markendexample\end{example}

\section{Extensions}
\label{s:extensions}

The framework for bounding extreme events presented in \cref{s:bounds-with-afs} can be extended in several ways. {Here we briefly summarize two extensions. Both are covered by the measure-theoretic approach of~\cite{Vinter1978,Vinter1978a,Lewis1980,Vinter1993}, but we give a more direct derivation.}

The first extension applies when upper bounds are sought on the maximum of $\Phi$ at a fixed finite time $T$, rather than its maximum over the time interval $[0,T]$. Such bounds can be proved by relaxing inequality~\cref{e:cond2} to require that $V$ bounds $\Phi$ only at time $T$. 

A second extension lets extreme events be defined using integrals over trajectories in addition to instantaneous values. Precisely, suppose the quantity we want to bound from above is
\begin{equation}
	\label{e:integral-cost}
	\sup_{\substack{x_0 \in X_0\\t \in \T}} \left\{\Phi[t,x(t;t_0,x_0)] + \int_{t_0}^t \Psi[\xi,x(\xi;t_0,x_0)] \dxi \right\}
\end{equation}
with chosen $\Phi$ and $\Psi$. One way to proceed is to augment the original dynamical system~\cref{e:system} with the scalar ODE $\dot{z} = \Psi(t,x)$, $z(t_0)=0$. Bounding~\cref{e:integral-cost} along trajectories of the original system is equivalent to bounding the maximum of $\Phi(t,x)+z$ pointwise in time along trajectories of the augmented system, and this can be done with the methods described in the previous sections. Another way to bound~\cref{e:integral-cost}, without introducing an extra ODE, is to replace condition~\cref{e:cond1} with
\begin{equation}
	\label{e:condition-integral}
	\mathcal{L}V(t,x) + \Psi(t,x) \leq 0 \quad\forall (t,x) \in \Omega.
\end{equation}
Minor modification to the argument leading to~\cref{e:weak-duality} proves that
\begin{equation}
	\label{e:bounds-integral}
	\sup_{\substack{x_0 \in X_0\\t \in \T}} \left\{\Phi[t,x(t;t_0,x_0)] + \int_{t_0}^t \Psi[\xi,x(\xi;t_0,x_0)] \dxi \right\} \leq 
	\inf_{\subalign{V:\,&\cref{e:cond2}\\&\cref{e:condition-integral}}} \sup_{x_0 \in X_0} \, V(t_0,x_0).
\end{equation}
As in \cref{e:weak-duality}, the righthand minimization is a convex problem and can be tackled computationally using SOS optimization for polynomial ODEs when $\Phi$ and $\Psi$ are polynomial. Analogues of~\cref{th:strong-duality,th:sos-convergence} for~\cref{e:bounds-integral} hold if $\Psi$ is continuous.

\section{Conclusions}
\label{s:conclusion}

We have {discussed} a convex framework for constructing \textit{a priori} bounds on extreme events in nonlinear dynamical systems governed by ODEs or PDEs. Precisely, we have {described} how to bound from above the maximum value $\maxphi$ of an observable $\Phi(t,x)$ over a given finite or infinite time interval, among all trajectories that start from a given initial set. This approach, {which is a particular case of general relaxation frameworks for optimal control and optimal stopping problems~\cite{Lewis1980,Cho2002}}, relies on the construction of auxiliary functions $V(t,x)$ that decay along trajectories and bound $\Phi$ pointwise from above. These constraints amount to the pointwise inequalities (\ref{e:V-conditions}a,b) in time and state space, which can be either imposed globally or imposed locally on any spacetime set that contains all trajectories of interest. Suitable global or local $V$ can be constructed without knowing any system trajectories, so $\maxphi$ can be bounded above even when trajectories are very complicated. We have given a range of ODE examples in which analytical or computational constructions give very good and sometimes sharp bounds. As a PDE example, we have proved analytical upper bounds on a quantity called fractional enstrophy for solutions to the one-dimensional Burgers equation with fractional diffusion.

The convex minimization of upper bounds on $\maxphi$ over global or local auxiliary functions is dual to the non-convex maximization of $\Phi$ along trajectories. In the case of ODEs and local auxiliary functions, \cref{th:strong-duality}, {which is a corollary of Theorem 2.1 and equation (5.3) in~\cite{Lewis1980}, guarantees that this duality is strong when the time interval is finite and the ODE satisfies certain continuity and compactness assumptions.} This means that the infimum over bounds is equal to the maximum over trajectories, so there exist $V$ proving arbitrarily sharp bounds on $\maxphi$. Further, strong duality holds in several of our ODE examples to which the assumptions of \cref{th:strong-duality} do not apply, including formulations with global $V$ or infinite time horizons. {However, neither the proofs in~\cite{Lewis1980} nor our alternative proof in \cref{s:direct-proof-strong-duality} can be easily extended to these cases because they rely on compactness, and} we have given counterexamples to strong duality with infinite time horizon even when trajectories remain in a compact set. Better characterizing the dynamical systems for which strong duality holds remains an open challenge.

Regardless of whether duality is weak or strong for a given dynamical system, constructing auxiliary functions that yield good bounds often demands ingenuity. Fortunately, as described in \cref{s:sos-optimization}, computational methods of sum-of-squares (SOS) optimization can be applied in the case of polynomial ODEs with polynomial $\Phi$. Moreover, \cref{th:sos-convergence} guarantees that if strong duality and mild compactness assumptions hold, then bounds computed by solving the SOS optimization problem~\cref{e:sos-opt} become sharp as the polynomial degree of the auxiliary function $V$ is raised.
In practice, computational cost can become prohibitive as either the dimension of the ODE system or the polynomial degree of $V$ increases, at least with the standard approach to SOS optimization wherein generic semidefinite programs are solved by second-order symmetric interior-point algorithms. For instance, given a 10-dimensional ODE system with no symmetries to exploit, the degree of $V$ is currently limited to about 12 on a large-memory computer. Larger problems may be tackled using specialized nonsymmetric interior-point~\cite{Papp2019} or first-order algorithms~\cite{Zheng2017csl, Zheng2018}. One also could replace the weighted SOS constraints in~\cref{e:sos-opt} with stronger constraints that may give more conservative bounds at less computational expense~\cite{AAAhmadi2015, AAAhmadi2019}.

In the case of PDEs, the bounding framework of \cref{s:bounds-with-afs} can produce valuable bounds, as in \cref{ex:fractional-burgers}, but theoretical results and computational tools are {lacking. \Cref{th:strong-duality}, which guarantees arbitrarily sharp bounds for many ODEs, does not apply to PDEs, nor can we directly apply the computational methods of \cref{s:sos-optimization} that work well for polynomial ODEs.} On the theoretical side, guarantees that feasible auxiliary functions exist for PDEs would be of great interest, not least because bounds on certain extreme events can preclude loss of regularity. {Statements formally dual to results in~\cite{Cho2002} for optimal stopping problems would imply that near-optimal auxiliary functions exist for autonomous PDEs, at least when extreme events occur at finite time, but such statements have not yet been proved.} On the computational side, constructions of optimal $V$ for PDEs would be very valuable, both to guide rigorous analysis and to improve on conservative bounds proved by hand. Methods of SOS optimization can be applied to PDEs in two ways. The first is to approximate the PDE as an ODE system and bound the error this incurs, obtaining an ``uncertain'' ODE system to which standard SOS techniques can be applied~\cite{Goulart2012, Chernyshenko2014, Huang2015, Goluskin2019}. The second approach is to work directly with the PDE using either the integral inequality methods of~\cite{Valmorbida2015tac, Valmorbida2015intsostools, Valmorbida2015cdc} or the moment relaxation techniques of~\cite{Korda2018,Marx2018}. These strategies have been used to study PDE stability, time averages, and optimal control, but they are in relatively early development. They have not yet been applied to extreme events as studied here, although the method in~\cite{Korda2018} applies to extreme behavior at a fixed time and could be extended to time intervals. It remains to be seen whether any of these strategies can numerically optimize auxiliary functions for PDEs of interest at reasonable computational cost, but recent advances in optimization-based formulations and corresponding numerical algorithms give us hope that this will be possible in the near future.

\currentpdfbookmark{Acknowledgements}{bookmark:acknowledgements}
\section*{Acknowledgments}
We are indebted to Andrew Wynn, Sergei Chernyshenko, Ian Tobasco, and Charles Doering, who offered many insightful comments on this work. We also thank the anonymous referees for comments that considerably improved the original version of this work.

\appendix

\section{Optimality of the quadratic \texorpdfstring{$V$}{V} in~\texorpdfstring{\cref{ex:nonautonomous-example-sos}}
	{Example 2.1}}
\label{app:nonautomonous-system-example-optimality}
The $V$ given by~\cref{e:nonautomonous-system-example-V} is optimal among all quadratic global auxiliary functions that produce upper bounds on $\Phi=x_1$ along the trajectory starting from the point $(0,1)$. 
To prove this, consider a general quadratic global auxiliary function,
\begin{multline}
	V(t,x_1,x_2) = C_0 + C_1 x_1 + C_2 x_2 + C_3 t 
	\\+ C_4 x_1^2 + C_5 x_2^2 + C_6 t^2 
	+ 2C_7 x_1 x_2 + 2C_8 t x_1 + 2C_9 t x_2.
\end{multline}
The coefficients $C_0,\,\ldots,\,C_9$ must be chosen to minimize the bound $\maxphi \leq V(0,0,1)$ implied by~\cref{e:weak-duality-incomplete}, subject to the inequality constraints~(\ref{e:V-conditions}a,b). Differentiating $V$ along solutions of~\cref{e:nonautomonous-system-example} yields
\begin{align}
	\mathcal{L}V(t,x_1,x_2) =\; 
	&C_3 + 
	(2C_9 - C_2) x_2 + (2C_8 - 0.1C_1) x_1 + 2C_6 t
	+ (C_2 - 0.2C_4)x_1^2
	\\
	&-(2.2C_7 + C_1)x_1 x_2 - 2C_5 x_2^2
	+(C_1-2C_9) t x_2 - (C_2+0.2C_8) t x_1
	\notag\\
	&+2C_7 x_1^3 - 2C_7 x_1 x_2^2 + 2(C_5 - C_4) x_1^2 x_2 + 2 C_7 t x_2^2
	\notag\\
	&+ 2(C_4-C_5-C_8) t x_1 x_2 + 2(C_9-C_7) t x_1^2 - 2 C_9 t^2 x_1 + 2C_8 t^2 x_2.
	\notag
\end{align}
In order for this expression to be nonpositive for all $(x_1,x_2)\in\R^2$ and $t\geq 0$, as required by~\cref{e:cond1}, the indefinite cubic terms and the quadratic terms proportional to $t$ must vanish. This forces us to set $C_1,C_2,C_7,C_8,C_9=0$ and $C_4=C_5$, so the expressions for $V$ and $\mathcal{L}V$ reduce to
\begin{subequations}
	\begin{gather}
		V(t,x_1,x_2) = C_0 + C_3 t + C_6 t^2 + C_5 \left( x_1^2 + x_2^2 \right),\\
		\mathcal{L}V(t,x_1,x_2) = C_3 + 2 C_6 t - 0.2C_5x_1^2- 2C_5 x_2^2.
	\end{gather}
\end{subequations}
Condition~\cref{e:cond1}, which requires $\mathcal{L}V\le0$, is satisfied only if $C_3,C_6\leq 0$ and $C_5 \geq 0$. With $\Phi=x_1$ condition~\cref{e:cond2} becomes
$C_0 - x_1 + C_5 x_1^2  + C_3 t + C_6 t^2 + C_5 x_2^2 \geq 0$, which in turn requires $4 C_0 C_5 \geq 1$. Minimizing the bound
$\maxphi \leq V(0,0,1) = C_0 + C_5$ under these constraints yields $C_0,C_5 = \tfrac12$, and we are free to choose any $C_3,C_6\leq 0$. Any such $V$ is optimal, including~\cref{e:nonautomonous-system-example-V} which results from choosing $C_3,C_6=0$.

\section{Sharp bounds for nonzero initial conditions in \texorpdfstring{\cref{ex:ex-1d-unbounded-trajectories}}{example~\ref{ex:ex-1d-unbounded-trajectories}}}
\label{app:sharp-bounds-ex-x2}
Auxiliary functions that give sharp bounds on $\Phi=4x/(1+4x^2)$ along single trajectories of the ODE~\cref{e:xdot=x2} exist for every nonzero initial condition $x_0$. Here we give global $V$, which also are local $V$ on any $\Omega$ in which trajectories remain. In the $x_0>0$ case, a global $V$ giving sharp upper bounds on $\maxphi_\infty$ is
\begin{equation}
	\label{e:V-example-blowup-solutions}
	V(t,x) = \begin{cases}
		1, &x \leq \tfrac12,\\
		\dfrac{4x}{1+4x^2}, &x>\tfrac12.
	\end{cases}
\end{equation}
This function is continuously differentiable and satisfies~(\ref{e:V-conditions}a,b). It is optimal because the bound on $\maxphi_\infty$ implied by~\cref{e:weak-duality} with $X_0 = \{x_0\}$ is
\begin{equation}
	\maxphi_\infty \leq V(0,x_0) = \begin{cases}
		1, &0 < x_0 \leq \tfrac12,\\
		\dfrac{4x_0}{1+4x_0^2}, &x_0>\tfrac12,
	\end{cases}
\end{equation}
which coincides with the expression~\cref{e:maxphi-example-unbounded-trajectories} for $\maxphi_\infty$.

The $x_0<0$ case requires a more complicated construction. An argument similar to that in \cref{ex:ex-1d-unbounded-trajectories} shows that any global optimal $V$ providing the sharp bound $\maxphi_\infty \leq 0$ must be time-dependent. The same is true for local $V$ unless $\Omega \subseteq [0,\infty) \times (-\infty,0]$, in which case $V=0$ is optimal. 
To construct a time-dependent global $V$ that is optimal for $X_0 = \{x_0\}$ with $x_0$ negative, we note that $\beta(t)=x_0/(1-x_0 t)$ solves the ODE~\cref{e:xdot=x2} with initial condition $x(0)=x_0$. Observe that $\beta(0)=x_0$, $\beta(t)<0$, and $\beta'(t)=\beta(t)^2$. Consider
\begin{equation}
	\rho(x) = \begin{cases}
		\displaystyle\exp\left(1-\frac{1}{1-x^2}\right), &\abs{x}<1,\\
		0, & \abs{x} \geq 1,
	\end{cases}	
\end{equation}
which is a smooth nonnegative function. We claim that
\begin{equation}
	\label{e:xdot=x2-Vopt}
	V(t,x) := \begin{cases}
		\displaystyle\rho\!\left(\frac{x}{\beta(t)}\right), &x\leq 0,\\
		1, &x>0
	\end{cases}
\end{equation}
is an optimal global auxiliary function. This $V$ implies the sharp bound $\maxphi_\infty \leq V(0,x_0)=0$ since $\rho(1)=0$, so it remains only to check~(\ref{e:V-conditions}a,b).
Inequality~\cref{e:cond2} holds because $\Phi$ is nonpositive for $x\leq 0$ and is bounded above by 1 pointwise. To verify~\cref{e:cond1}, we consider positive and nonpositive $x$ separately. The $x>0$ case is immediate because $\mathcal{L}V(t,x)=0$. For $x\leq 0$, a straightforward calculation using $\beta'(t) = \beta(t)^2$ gives
\begin{align}
	\label{e:xdot=x2-LV}
	\mathcal{L}V(t,x)
	= \partial_t V + x^2 \,\partial_x V
	= \frac{x}{\beta(t)}\left[x - \beta(t) \right]\, \rho'\!\left(\frac{x}{\beta(t)}\right).
\end{align}

\noindent
Observe that $\rho'(s)$ vanishes if $s=0$ or $\abs{s}\geq 1$, so $\mathcal{L}V=0$ if $x\leq \beta(t)$ or $x=0$. When $\beta(t)<x<0$ instead, $\mathcal{L}V<0$ because the first two factors in~\cref{e:xdot=x2-LV} are positive, while $\rho'(s)$ is negative for $0<s<1$. Combining these observations shows that $\mathcal{L}V\leq 0$ for all times if $x\leq 0$.  \Cref{f:sketch-optimal-V} illustrates the behavior of $V$ and $\mathcal{L}V$ when $x_0=-\frac34$.

\begin{figure}
	\centering
	\includegraphics[scale=1]{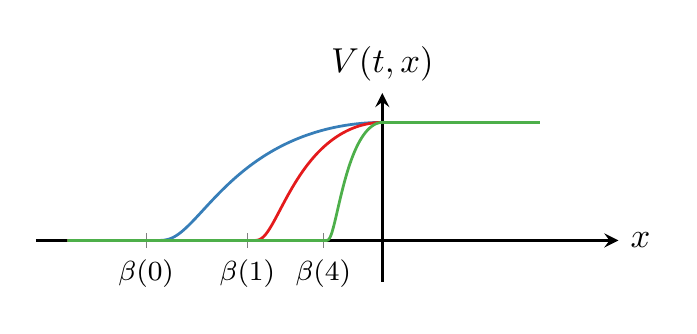} \hspace{2em}
	\includegraphics[scale=1]{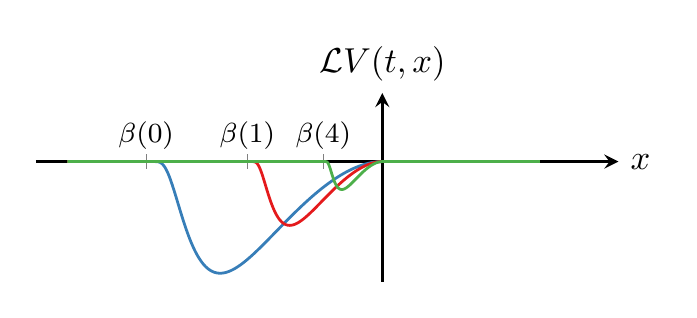}\\
	\includegraphics[scale=1]{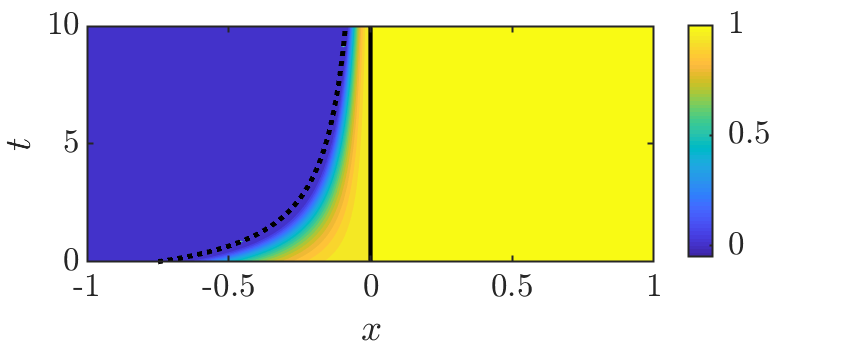} \hspace{2em}
	\includegraphics[scale=1]{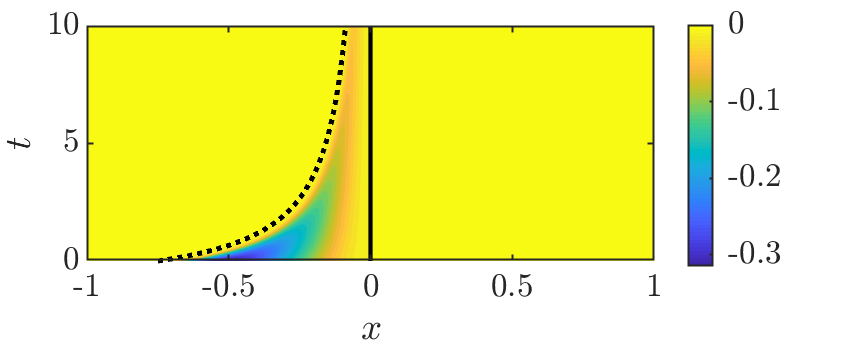}
	\caption{
		\textit{Top row:} Profiles of the auxiliary function $V(t,x)$ in~\cref{e:xdot=x2-Vopt} and its derivative along trajectories $\mathcal{L}V(t,x)$, plotted as a function of $x$ for $t=0$, $t=1$, and $t=4$. \textit{Bottom row:} Contours of $V$ (left) and $\mathcal{L}V$ (right). Lines mark the trajectory $x=\beta(t)$ (\dashedrule) and the semistable equilibrium $x=0$ (\solidrule). Outside the region between these two lines, $\mathcal{L}V=0$.
		All plots are for $x_0=-\frac34$.
	}
	\label{f:sketch-optimal-V}
\end{figure}

\section{Improving bounds iteratively with polynomial \texorpdfstring{$V$}{V} of fixed degree}
\label{app:iterative-procedure}
{Bounds computed with~\cref{e:sos-opt} can be improved without increasing the degree $d$ by using an iterative procedure. First, solve~\cref{e:sos-opt} to obtain an upper bound $\maxphi \leq \lambda_{d,0}^*$, which implies $\Phi(t,x) \leq \lambda_{d,0}^*$ along trajectories of interest. Then, replace the original set $\Omega$ in which trajectories remain with its subset $\Omega_1 := \Omega \cap \{(t,x): \Phi(t,x) \leq \lambda_{d,0}^*\}$. Since $\Omega_1 \subseteq \Omega$ is still basic semialgebraic, one can solve~\cref{e:sos-opt} again, but with the WSOS constraints defined on $\Omega_1$ rather than $\Omega$. This produces a new bound, $\maxphi \leq \lambda_{d,1}^* \leq \lambda_{d,0}^*$. The process can be iterated by taking $\Omega_{i+1} = \Omega \cap \{(t,x): \Phi(t,x) \leq \lambda_{d,i}^*\}$, $i=1,\,2,\,\ldots$, until the bound on $\maxphi$ stops improving. The WSOS optimization problem to be solved for each $i$ has constant computational cost, which is higher than the original one but typically much smaller than solving~\cref{e:sos-opt} with larger~$d$.}

\begin{table}[t]
	\caption{Upper bounds on $\maxphiinf$ for \cref{ex:sos-2d-example}, computed using time-independent polynomial auxiliary functions $V(x)$ of degree $d$ by the iterative procedure described in \cref{app:iterative-procedure}.}
	\label{t:sos-2d-example-bounds-iterative}
	\centering
	\small
	\begin{tabular}{ccccccccccccc}
		\toprule
		Iteration && $d=4$ && $d=6$ &&$d=8$ && $d=10$ &&$d=12$ &&$d=14$\\[2pt]
		\cline{1-1} \cline{3-3} \cline{5-5} \cline{7-7} \cline{9-9} \cline{11-11} \cline{13-13}\\[-2ex]
		1 && 2.194343 && 1.942396 && 1.931330 && 1.916228 && 1.903525 && 1.903448\\
		2 && 2.194343 && 1.934692 && 1.926088 && 1.913889 && 1.903346 && 1.903307\\
		3 && 2.194343 && 1.934643 && 1.926088 && 1.913817 && 1.903280 && 1.903250\\
		4 && 2.194342 && 1.934642 && 1.926086 && 1.913815 && 1.903260 && 1.903222\\
		5 && 2.194342 && 1.934642 && 1.926086 && 1.913814 && 1.903249 && 1.903207\\
		\bottomrule
	\end{tabular}
\end{table}

{\Cref{t:sos-2d-example-bounds-iterative} reports bounds on $\maxphiinf$ obtained with this iterative procedure for the problem described in \cref{ex:sos-2d-example}, using polynomial $V$ of degrees up to 14. Each iteration lowers the bound as expected. The improvement with each iteration is small in this example, especially with lower-degree $V$. Raising $d$ by 2 offers much more improvement except when the bound is nearly sharp already. It remains to be tested whether the iterative scheme brings more gains for other problems.}

\section{An elementary proof of \texorpdfstring{\cref{th:strong-duality}}{Theorem~\ref{th:strong-duality}}}
\label{s:direct-proof-strong-duality}
Under the assumptions of  \cref{th:strong-duality}, differentiable auxiliary functions that produce arbitrarily sharp bounds on~$\maxphi_T$ can be constructed by approximating the optimal but generally discontinuous $V^*$ defined in \cref{ss:discontinuous-afs}. {This construction, which resembles the argument in~\cite{Hernandez1996}, yields \cref{th:strong-duality} without the measure theory or convex analysis used in the proofs of~\cite{Lewis1980}.}

\subsection{Construction of near-optimal $V$} Let $\delta > 0$. We must show that there exists a $C^1$ function $V$ on $\Omega = [t_0,T]\times X$ that satisfies (\ref{e:V-conditions}a,b) and
\begin{equation}
	\label{e:delta-suboptimal-bound}
	\sup_{x_0 \in X_0} V(t_0,x_0) \leq  \maxphiT + \delta.
\end{equation}
To do this we construct $W\in C^1(\Omega)$ such that
\begin{subequations}
	\label{e:V-conditions-relaxed}
	\begin{align}
		\label{e:cond1-W}
		\mathcal{L}W(t,x) &\leq \frac{\delta}{5(T-t_0)} \hspace{15pt}\qquad\text{on } \Omega,\\
		\label{e:cond2-W}
		\Phi(t,x)  &\leq  W(t, x) + \frac{2}{5}\delta \qquad\text{on } \Omega, \\
		\label{e:cond3-W}
		\sup_{x_0 \in X_0} W(t_0,x_0) &\leq  \maxphiT + \frac{2}{5}\delta.
	\end{align}
\end{subequations}
Then,~(\ref{e:V-conditions}a,b) and~\cref{e:delta-suboptimal-bound} are satisfied by the continuously differentiable function
\begin{equation}
	V(t,x) := W(t,x) + \frac{2}{5}\delta + \frac{(T-t)\delta}{5(T-t_0)}.
\end{equation}

Our construction of $W$ uses the flow map $S_{(s,t)}: Y \to \R^n$, defined for any two fixed time instants $s$ and $t$ such that $t_0 \leq s \leq t \leq t_1$ as $S_{(s,t)} y = x(t \given s, y).$ In other words, $S_{(s,t)}y$ is the point at time~$t$ on the trajectory of the ODE $\dot{x}=F(\xi,x)$ that passed through~$y$ at time~$s$. An explicit expression for the flow map is generally not available. Nonetheless, under the assumptions of \cref{th:strong-duality}, the flow map is well defined and satisfies
\begin{subequations}
	\begin{gather}
		\label{e:flow-integration}
		S_{(s,t)}y = y + \int_s^t F[\xi, S_{(s,\xi)}y] \dxi,
		\\
		\label{e:group-property}
		S_{(s,t)}  \circ S_{(r,s)} = S_{(r,t)} \qquad \forall r,t,s:\,
		t_0 \leq r \leq s \leq t.
	\end{gather}
\end{subequations}
The function $(t,s,y) \mapsto S_{(s,t)}y$ is uniformly continuous with respect to both $s$ and $y$ for $t$ in compact time intervals; see, for instance,~\cite[Chapter V, Theorem 2.1]{Hartman2002}. It also is locally Lipschitz in the sense of the following Lemma, which is proved in \cref{app:proof-lemma-flow}.
\begin{lemma}
	\label{lemma:lemma-flow}
	Suppose the assumptions of \cref{th:strong-duality} hold and let $[a,b] \times K$ be a compact subset of $[t_0,t_1]\times Y$.
	There exist positive constants $C_1$ and $C_2$, dependent only on $a$, $b$, $K$, $t_0$ and $t_1$, such that:
	\begin{enumerate}[(i)]
		\item $\|S_{(t,\xi)} x - S_{(t,\xi)} y\| \leq C_1 \|x-y\|$ for all $x,y \in K$, all $t\in[a,b]$, and all {$\xi \in [t,t_1]$}.
		\item $\|S_{(t,\xi)} x - S_{(s,\xi)} x\| \leq C_2 \abs{t - s}$ for all $x \in K$, all $t,s \in [a,b]$, and all $\xi \in[\max(t,s),t_1]$.
	\end{enumerate}
\end{lemma}

We also need the following Lemma, proved in \cref{app:proof-U-function}, which states that the optimal but possibly discontinuous auxiliary function defined by~\cref{e:value-function} can be approximated by a locally Lipschitz function.
\begin{lemma} 
	\label{lemma:lemma-U-function}
	There exist {$t_2 \in (T,t_1)$ and} a locally Lipschitz function $U : [t_0,t_2]\times Y \to \R$ that satisfies
	\begin{subequations}
		\label{e:U-conditions}
		\begin{gather}
			\label{e:U-cond-1}
			\Phi(t,x) \leq U(t,x) + \frac{\delta}{5}  \quad\text{on } \Omega,\\
			\label{e:U-cond-3}
			\sup_{x_0 \in X_0} U(t_0,x_0) \leq \maxphiT + \frac{\delta}{5},
		\end{gather}
		and, for each fixed
		{$(t,x) \in [t_0,t_2)\times Y$}, 
		\begin{equation}
			\label{e:U-cond-2}
			U(t+\varepsilon,S_{(t,t+\varepsilon)}x)  \leq U(t,x) \qquad \forall \varepsilon \in (0,{t_2}-t).
		\end{equation}
	\end{subequations}
\end{lemma}

A function $W \in C^1(\Omega)$ that satisfies~(\ref{e:V-conditions-relaxed}a,b,c) can be constructed by mollifying $U$ ``{forward} in time'' on $\Omega$. Precisely, fix any nonnegative differentiable mollifier $\rho(t,x)$ that is supported on the closed unit ball of $\mathbb{R}\times \mathbb{R}^n$ and has unit integral. For each $k\ge1$ define
\begin{equation}
	\rho_k(t,x) := k^{n+1} \rho(kt{+}1,kx).
\end{equation}
Observe that $\rho_k$ is supported on $R_k = {[-2k^{-1},0]} \times B_n(0,k^{-1})$, where $B_n(0,r)$ denotes the closed $n$-dimensional ball of radius $r$ centered at the origin, and has unit integral. Let $k$ be large enough that ${[t_0,t_2]}\times Y$ contains the compact set
\begin{equation}
	\mathcal{N} 
	\phantom{:}= \{ {(t - s,x - y)}:\,(t,x) \in \Omega,\,(s,y) \in R_k \}.
	\label{e:N-neighborhood}
\end{equation}
{Note that $\Omega \subset \mathcal{N}$.} For each $(t,x) \in \Omega$, define
\begin{equation}
	W(t,x) 
	:= (\rho_k \ast U)(t,x)
	= \int_{R_k} \rho_k(s,y) U(t-s,x-y) \,{\rm d}s \, {\rm d}^ny.
	\label{e:Wdef}
\end{equation}
Since $R_k$ contains only {nonpositive times $s\leq 0$}, $W$ is a {forward}-in-time mollification of $U$. Standard arguments~\cite[Appendix C.4]{Evans1998} show that $W$ is continuously differentiable on $\Omega$. Because $\Omega$ is compact and $U$ is continuous, $W\to U$ uniformly on $\Omega$ as $k\to\infty$. Thus we can choose $k$ large enough to ensure
\begin{equation}
	\label{e:norm-U-W}
	\|U - W\|_{C^0(\Omega)} \leq \frac{\delta}{5},
\end{equation}
To see that $W$ satisfies~\cref{e:cond3-W}, combine~\cref{e:norm-U-W} with~\cref{e:U-cond-3} to estimate
\begin{equation}
	\sup_{x_0 \in X_0} W(t_0,x_0)
	\leq \sup_{x_0 \in X_0} U(t_0,x_0)  + \|U - W\|_{C^0(\Omega)}
	\leq \maxphiT + \frac{2}{5}\delta.
\end{equation}
We similarly obtain~\cref{e:cond2-W} by estimating the righthand side of~\cref{e:U-cond-1} as
\begin{equation}
	\Phi(t,x) 
	\leq U(t,x) + \frac{\delta}{5}
	\leq W(t,x) + \|U - W\|_{C^0(\Omega)} + \frac{\delta}{5}
	\leq W(t,x) + \frac{2}{5}\delta.
\end{equation}
To prove~\cref{e:cond1-W}, fix $(t,x) \in \Omega$ and bound
\begin{align}
	\label{e:LW-estimate-1}
	\mathcal{L}W(t,x)
	&= \lim_{\varepsilon \to 0} \frac{W(t+\varepsilon,S_{(t,t+\varepsilon)}x) - W(t,x)}{\varepsilon} \\[1ex]
	&= \lim_{\varepsilon \to 0} \frac{1}{\varepsilon} \int_{R_k} \rho_k(s,y) 
	\left[ U(t+\varepsilon-s,S_{(t,t+\varepsilon)}x-y) 
	- U(t-s,x-y)
	\right]{\rm d}s \, {\rm d}^ny
	\notag \\[1ex]
	&\leq \lim_{\varepsilon \to 0} \frac{1}{\varepsilon} \int_{R_k} \rho_k(s,y) 
	\left\{ U(t+\varepsilon-s,S_{(t,t+\varepsilon)}x-y)  \right.
	\notag \\[-0.5ex]
	& \left.\hspace{150pt}- U[t+\varepsilon-s,S_{(t-s,t-s+\varepsilon)}(x-y)]
	\right\} {\rm d}s \, {\rm d}^ny
	\notag \\[1ex]
	&\leq \lim_{\varepsilon \to 0} \frac{C}{\varepsilon} \int_{R_k} \rho_k(s,y) 
	\left\| S_{(t,t+\varepsilon)}x - y - S_{(t-s,t-s+\varepsilon)}(x-y) \right\| {\rm d}s \, {\rm d}^ny,
	\notag
\end{align}
where $C$ is a positive constant independent of $t$ and $x$. The two inequalities above follow, respectively, from~\cref{e:U-cond-2} and the uniform Lipschitz continuity of $U$ on compact sets.

Since {$t \leq T < t_2$}, forward-in-time trajectories are well defined for sufficiently small $\varepsilon$. Moreover, reasoning as in the proof of \cref{lemma:lemma-flow} in \cref{app:proof-lemma-flow} shows that trajectories starting from the compact neighborhood $\mathcal{N}$ of $\Omega$ defined in~\cref{e:N-neighborhood} are uniformly bounded up to time {$t_2$}. Thus the rightmost integrand in~\cref{e:LW-estimate-1} is uniformly bounded and, by the dominated convergence theorem, we can exchange the limit and the integral. Then, we can further estimate $\mathcal{L}W$ using the fact that $\rho_k$ has unit integral over $R_k$, the relation~\cref{e:flow-integration}, and the mean value theorem:
\begin{align}
	\mathcal{L}W(t,x) 
	&\leq C 
	\adjustlimits
	\max_{(s,y)\in R_k} 
	\lim_{\varepsilon \to 0} \frac{1}{\varepsilon} \left\| S_{(t,t+\varepsilon)}x - y - S_{(t-s,t-s+\varepsilon)}(x-y) \right\|
	\\[0.5ex]
	&=C 
	\adjustlimits
	\max_{(s,y)\in R_k} 
	\lim_{\varepsilon \to 0} \frac{1}{\varepsilon} \left\| 
	\int_t^{t+\varepsilon} F(\xi, S_{(t,\xi)}x) \dxi 
	- \int_{t-s}^{t-s+\varepsilon} F[\xi, S_{(t-s,\xi)}(x-y)] \dxi 
	\right\|
	\notag \\[0.5ex]
	&=C \max_{(s,y)\in R_k} 
	\left\|  F(t,x) - F(t-s, x-y) \right\|.
	\phantom{\int_t^{t+\varepsilon}}	
	\notag
\end{align}
Both $(t,x)$ and $(t-s,x-y)$ lie in the compact set $\mathcal{N}$. Since $F$ is locally Lipschitz by assumption, it is uniformly Lipschitz on $\mathcal{N}$. Consequently, there exist a constant $C'$, independent of $t$ and $x$, and a $k$ sufficiently large such that
\begin{equation}
	{\mathcal{L}W(t,x) \leq C' \max_{(s,y)\in R_k} 
		(\abs{s} + \left\|  y \right\|) = \frac{3 C'}{k} \leq \frac{\delta}{5(T-t_0)},}
	\label{e:LW-complete}
\end{equation}
meaning that $W$ satisfies~\cref{e:cond1-W} as claimed. This concludes the proof of \cref{th:strong-duality}.

\begin{remark}
	{Defining $\rho_k$ such that the mollification~\cref{e:Wdef} is forward in time, so $s \leq 0$ on $R_k$, is key to prove~\cref{e:LW-complete} for all $(t,x) \in \Omega=[t_0,T]\times X$. If $s>0$ anywhere on $R_k$, given any finite $k$ we would have $t-s < t_0$ for all $t \in [t_0,t_k]$ and some $t_k> t_0$. In this case, we would not have the first inequality in~\cref{e:LW-estimate-1} for all $(t,x) \in \Omega$ because~\cref{e:U-cond-2} holds only after time~$t_0$.}
\end{remark}

\subsection{Proof of \texorpdfstring{\cref{lemma:lemma-flow}}{Lemma~\ref{lemma:lemma-flow}}}
\label{app:proof-lemma-flow}
To establish part \textit{(i)} of \cref{lemma:lemma-flow}, observe that assumption (A.2) in \cref{th:strong-duality} guarantees that the trajectory starting from any $x \in K$ at any time $t \in [a,b]$ exists up to time {$t_1$}, so in particular $\|S_{(t,\xi)} x\|$ is bounded for all $\xi \in [t,{t_1}]$. Combining the compactness of $[a,b] \times K$ with the continuity of trajectories with respect to both the initial point and the initial time~\cite[Chapter V, Theorem 2.1]{Hartman2002} shows that trajectories are uniformly bounded in norm. Precisely, there exists a constant $M$, depending only on $a$, $b$, $K$ and ${t_1}$, such that $\|S_{(t,\xi)} x\| \leq M$ for all $(t,x) \in [a,b] \times K$ and all $\xi \in [t,{t_1}]$. We therefore can apply Lemma 2.9 from~\cite{Robinson2001} and the local Lipschitz continuity of $F(\cdot,\cdot)$ to find a constant $\Lambda_1$, dependent only on $a$, $b$ and $K$, such that
\begin{equation}
	\frac{{\rm d}}{{\rm d} \xi} \| S_{(t,\xi)} x - S_{(t,\xi)} y \| 
	\leq \|F(\xi,S_{(t,\xi)} x) - F(\xi,S_{(t,\xi)} y)\|
	\leq \Lambda_1 \| S_{(t,\xi)} x - S_{(t,\xi)} y \|
\end{equation}
for all $x,y \in K$, all $t \in [a,b]$, and all $\xi \in [t,{t_1}]$. Assertion \textit{(i)} then follows with $C_1 = {\rm e}^{\Lambda_1 {t_1}}$ after applying Gronwall's inequality to bound
\begin{equation}
	\| S_{(t,\xi)} x - S_{(t,\xi)} y \| \leq {\rm e}^{\Lambda_1 \xi} \| x - y \| \leq {\rm e}^{\Lambda_1 {t_1}} \| x - y \|.
\end{equation}

To prove part \textit{(ii)} of \cref{lemma:lemma-flow}, assume without loss of generality that $s<t$. For all $\xi \in[t,{t_1}]$, identity~\cref{e:group-property} gives
$\| S_{(t,\xi)} x - S_{(s,\xi)} x \| = 
\| S_{(t,\xi)} x - S_{(t,\xi)}S_{(s,t)} x\|.
$
Proceeding as above with $y=S_{(s,t)} x$ shows that
\begin{equation}
	\label{e:preliminary-estimate-lemma-flow}
	\| S_{(t,\xi)} x - S_{(s,\xi)} x \| \leq \Lambda_2 \left\| x - S_{(s,t)}x \right\|
\end{equation}
for some positive constant $\Lambda_2$.
Moreover, we can use~\cref{e:flow-integration} to estimate
\begin{equation}
	\left\| S_{(s,t)}x - x \right\| 
	= \left\| \int_s^t F(\xi, S_{(s,\xi)}x ) \,{\rm d}\xi \right\| 
	\leq \sqrt{n} \int_s^t \| F(\xi,S_{(s,\xi)} x)\| \,{\rm d}\xi.
\end{equation}
Since $F$ is continuous and, as noted above, $\|S_{(s,\xi)} x\| \leq M$ for all $(s,x) \in [a,b] \times K$ and all $\xi \in [s,{t_1}] \subset [a,{t_1}]$,
\begin{equation}
	\left\| S_{(s,t)}x - x \right\| \leq \sqrt{n} \max_{\substack{\xi \in [a,{t_1}] \\ \|y\|\leq M}} \| F(\xi,y) \| \, \abs{t-s}.
\end{equation}
Combining this with~\cref{e:preliminary-estimate-lemma-flow} proves the claim for a suitable choice of $C_2$.

\subsection{Proof of \texorpdfstring{\cref{lemma:lemma-U-function}}{Lemma~\ref{lemma:lemma-U-function}}}
\label{app:proof-U-function}
Fix $t_2 = T + \gamma$ for some $\gamma>0$ sufficiently small and to be determined later. Arguing as in the proof of \cref{lemma:lemma-flow}\textit{(i)}, trajectories starting from $x_0 \in X_0$ remain bounded uniformly in the initial condition and time. Precisely, there exists a constant $M$ such that $\|S_{(t_0,t)}x_0\|\leq M$ for all $x_0 \in X_0$ and $t \in [t_0,t_2]$. If $\mathcal{B}$ denotes the $n$-dimensional ball of radius $M$ centered at the origin, we conclude that the compact set $[t_0,t_2] \times \mathcal{B}$  contains $\Omega = [t_0,T]\times X$, the spacetime set in which trajectories starting from $x_0 \in X_0$ at time $t_0$ remain up to time $T$.

Let $\Psi: \mathbb{R}\times \mathbb{R}^n \times Y \to \R$ be a Lipschitz approximation of $\Phi$ satisfying
\begin{equation}
	\label{e:phi-approximation-psi}
	\|\Phi - \Psi\|_{C^0([t_0,t_2] \times \mathcal{B})} \leq \frac{\delta}{10}.
\end{equation}
Such $\Psi$ may be constructed in a number of ways, for instance by using the Stone--Weierstrass theorem to approximate $\Phi$ uniformly on the compact set $[t_0,t_2] \times \mathcal{B}$ by a polynomial, and extending such polynomial to a Lipschitz function on $\mathbb{R}\times \mathbb{R}^n$. We claim that $t_2$ can be chosen such that the function $U: {[t_0,t_2]}\times Y \to \mathbb{R}$ defined as
\begin{equation}
	\label{e:U-def}
	U(t,x) := \sup_{\tau \in [t,t_2]} \Psi[\tau, S_{(t,\tau)}x]
\end{equation}
satisfies~(\ref{e:U-conditions}a--c). This $U$ cannot be computed in practice but is well defined. Note that if $\Phi$ is Lipschitz we can choose $\Psi=\Phi$ and the restriction of $U$ to $\Omega$ tends to the optimal but possibly discontinuous auxiliary function defined in~\cref{e:value-function} as $\gamma = t_2 - T$ tends to zero. If $\gamma$ is finite but small, then $U$ approximates this optimal auxiliary function. The same is true when $\Psi$ only approximates $\Phi$.

To see that~\cref{e:U-cond-1} holds, note that $U(t,x) \geq \Psi(t,x)$. Since $\Omega \subset [t_0,t_2] \times \mathcal{B}$ we conclude from~\cref{e:phi-approximation-psi} that, for all $(t,x) \in \Omega$,
\begin{equation}
	\Phi(t,x) 
	\leq \Psi(t,x) + \|\Phi - \Psi\|_{C^0([t_0,t_2] \times \mathcal{B})}
	\leq U(t,x) + \frac{\delta}{10}
	< U(t,x) + \frac{\delta}{5}.
\end{equation}

To prove~\cref{e:U-cond-3}, we will choose $\gamma=t_2-T$ such that
\begin{equation}
	\label{e:U-bound-sufficient-cond3-A}
	U(t_0,x_0) = \sup_{\tau \in [t_0,t_2]} \Psi[\tau, S_{(t_0,\tau)}x_0] \leq \maxphiT + \frac{\delta}{5}
\end{equation}
uniformly in the initial condition $x_0 \in X_0$. To do this, fix $x_0 \in X_0$ and observe that the supremum over $\tau \in [t_0,t_2]$ must be attained because the function $\tau \mapsto \Psi[\tau, S_{(t_0,\tau)}x_0]$ is continuous. If the supremum is attained on the interval $[t_0,T]$, then
\begin{align}
	\label{e:U-bound-sufficient-cond3-B}
	\sup_{\tau \in [t_0,t_2]} \Psi[\tau, S_{(t_0,\tau)}x_0] 
	&= \sup_{\tau \in [t_0,T]} \Psi[\tau, S_{(t_0,\tau)}x_0] \\
	&\leq \sup_{\tau \in [t_0,T]} \Phi[\tau, S_{(t_0,\tau)}x_0] + \|\Phi - \Psi\|_{C^0([t_0,t_2] \times \mathcal{B})}
	\nonumber\\
	&\leq \maxphiT + \frac{\delta}{10}.
	\nonumber
\end{align}
Instead, if the supremum is attained at time $t^* \in [T,t_2]$, then we can use the Lipschitz continuity of $\Psi$, the group property~\cref{e:group-property} of the flow map, and \cref{lemma:lemma-flow}\textit{(ii)} to find constants $C$ and $C'$, dependent on $t_0$, $t_1$ and the set $X_0$ but not on the choice of $x_0 \in X_0$, such that
\begin{align}
	\label{e:U-bound-sufficient-cond3-C}
	\sup_{\tau \in [t_0,t_2]} \Psi[\tau, S_{(t_0,\tau)}x_0] 
	&= \Psi[t^*, S_{(t_0,t^*)}x_0]
	\\[-2ex]
	&\leq  \Psi[T, S_{(t_0,T)}x_0] + \abs{ \Psi[t^*, S_{(t_0,t^*)}x_0] - \Psi[T, S_{(t_0,T)}x_0] }
	\nonumber\\
	&\leq \Psi[T, S_{(t_0,T)}x_0]  + C\abs{t^*-T} + C \| S_{(T,t^*)} S_{(t_0,T)}x_0 - S_{(t_0,T)}x_0 \|
	\nonumber\\
	&\leq \Phi[T, S_{(t_0,T)}x_0] + \|\Phi - \Psi\|_{C^0([t_0,t_2] \times \mathcal{B})} + (C + C') \abs{t^*-T}
	\nonumber\\
	&\leq \maxphiT + \frac{\delta}{10} + (C + C') \gamma.
	\nonumber
\end{align}
Upon setting $\gamma = \delta/[10(C+C')]$,~\cref{e:U-bound-sufficient-cond3-B} and~\cref{e:U-bound-sufficient-cond3-C} prove that~\cref{e:U-bound-sufficient-cond3-A} holds uniformly in the initial condition $x_0$ irrespective of whether the sup over $\tau$ is attained before or after time $T$.

Finally, to obtain~\cref{e:U-cond-2}, fix $(t,x) \in {[t_0,t_2)}\times Y$ and observe that, for all $\varepsilon\in (0,{t_2}-t)$,
\begin{align}
	U(t+\varepsilon,S_{(t,t+\varepsilon)}x) 
	&= \sup_{\tau \in [t+\varepsilon,{t_2}]} \Psi[\tau, S_{(t+\varepsilon,\tau)}S_{(t,t+\varepsilon)}x] \\
	&= \sup_{\tau \in [t+\varepsilon,{t_2}]} \Psi[\tau, S_{(t,\tau)} x] \notag \\
	&\leq \sup_{\tau \in [t,{t_2}]} \Psi[\tau, S_{(t,\tau)} x] \notag \\
	&= U(t,x). \notag
\end{align}

To conclude the proof of \cref{lemma:lemma-U-function}, we must prove that $U$ is locally Lipschitz on ${[t_0,t_2]}\times Y$, meaning that for each compact subset $[a,b] \times K$ of ${[t_0,t_2]}\times Y$ there exists a constant $C$ (dependent only on $a$, $b$, $K$, {$t_0$}, and {$t_2$}) such that
\begin{equation}
	\label{e:U-lipschitz}
	\abs{U(t,x) - U(s,y)} \leq C \left( \abs{s-t} + \|x-y\| \right) \qquad \forall\,
	(t,x),(s,y) \in  [a,b] \times K.
\end{equation}
Clearly, it suffices to find constants $C'$ and $C''$ such that
\begin{subequations}
	\begin{align}
		\label{e:wt-minus-ws}
		U(t,x)- U(s,y)  \leq C' \left( \abs{t-s} + \| x -  y\| \right),\\
		U(s,y) -U(t,x)  \leq C'' \left( \abs{t-s} + \| x -  y\| \right),
		\label{e:ws-minus-wt}
	\end{align}
\end{subequations}
To simplify the presentation below, we let $C$ to denote any absolute constant; its value may vary from line to line. We also assume without loss of generality that $s\leq t$.

To prove~\cref{e:wt-minus-ws} observe that, since $s\leq t$,
\begin{align}
	\label{e:wt-minus-ws-estimate}
	U(t,x)- U(s,y) 
	&=
	\sup_{\tau \in [t,{t_2}]} \Psi[\tau, S_{(t,\tau)}x]
	- \sup_{\tau \in [s,{t_2}]} \Psi[\tau, S_{(s,\tau)}y]
	\\
	&\leq
	\sup_{\tau \in [t,{t_2}]} \Psi[\tau, S_{(t,\tau)}x]
	- \sup_{\tau \in [t,{t_2}]} \Psi[\tau, S_{(s,\tau)}y]
	\notag \\
	&\leq
	\sup_{\tau \in [t,t_2]} \left\{ 
	\Psi[\tau, S_{(t,\tau)}x]
	- \Psi[\tau, S_{(s,\tau)}y] \right\}. \notag
\end{align}
The term inside the last supremum can be bounded uniformly in $\tau$. The Lipschitz continuity of $\Psi$ and \cref{lemma:lemma-flow} imply
\begin{align}
	\Psi[\tau, S_{(t,\tau)}x] - \Psi[\tau, S_{(s,\tau)}y]
	&\leq C \|S_{(t,\tau)} x - S_{(s,\tau)} y\|
	\\
	&\leq C \|S_{(t,\tau)} x - S_{(t,\tau)} y\| + C \|S_{(t,\tau)} y - S_{(s,\tau)} y\|
	\notag \\
	&\leq C \left( \| x -  y\| + \abs{t-s} \right).
	\notag
\end{align}
Combining this estimate with~\cref{e:wt-minus-ws-estimate} yields~\cref{e:wt-minus-ws}. 

To show~\cref{e:ws-minus-wt} we seek an upper bound on
\begin{equation}
	U(s,y)- U(t,x) = 
	\sup_{\tau \in [s,t_2]} \Psi[\tau, S_{(s,\tau)}y] - 
	\sup_{\tau \in [t,t_2]} \Psi[\tau, S_{(t,\tau)}x].
\end{equation}
If the first supremum can be restricted to $[t,t_2]$ without affecting its value, then we proceed as before. Otherwise, we restrict the supremum to $[s,t]$ and estimate
\begin{equation}
	\label{e:ws-minus-wt-partial}
	U(s,y)- U(t,x) 
	\leq
	\sup_{\tau \in [s,t]}  \Psi[\tau, S_{(s,\tau)}y]
	- \Psi(t,x)
	=
	\sup_{\tau \in [s,t]} \left\{ 
	\Psi[\tau, S_{(s,\tau)}y]
	- \Psi(t,x)
	\right\}.
\end{equation}
As before, the term inside the supremum can be bounded uniformly in $\tau$ using Lipschitz continuity and \cref{lemma:lemma-flow}. Precisely, since $\tau \leq t$ and $S_{(\tau,\tau)}y=y$,
\begin{align}
	\Psi(\tau,S_{(s,\tau)} y) - \Psi(t,x)
	&\leq C \left( \abs{\tau-t} + \|S_{(s,\tau)} y - x\|\right)
	\\
	&\leq C \left( \abs{t-s} + \|S_{(s,\tau)} y - S_{(\tau,\tau)}y \| + \| y - x\|\right)
	\notag \\
	&\leq C \left( \abs{t-s} + \| x-y \|\right). \notag
\end{align}
Combining these estimates with~\cref{e:ws-minus-wt-partial} yields~\cref{e:ws-minus-wt}.

\pdfbookmark{References}{bookmark:references}
\bibliography{./references.bib}

\end{document}